\documentclass[twoside]{article}

\usepackage[accepted]{aistats2022_preprint}
\usepackage[latin1]{inputenc}
\usepackage{amssymb,dsfont,amsmath,amsthm,mathtools}
\usepackage{hyperref}
\usepackage{graphicx}
\usepackage{notations, mystyle}

\def\ones{\mathds{1}}

\usepackage{float}
\usepackage{algorithm}
\usepackage{algorithmic}

\begin{document}

%
\runningtitle{Faster Unbalanced Optimal Transport: Translation invariant Sinkhorn and 1-D Frank-Wolfe}

%

\twocolumn[

\aistatstitle{Faster Unbalanced Optimal Transport: \\ Translation invariant Sinkhorn and 1-D Frank-Wolfe}

\aistatsauthor{Thibault Sejourne \And Francois-Xavier Vialard \And Gabriel Peyre}

\aistatsaddress{DMA, ENS, PSL \And  LIGM, UPEM \And CNRS, DMA, ENS, PSL} ]


\begin{abstract}
	Unbalanced optimal transport (UOT) extends optimal transport (OT) to take into account mass variations to compare distributions.
	This is crucial to make OT successful in ML applications, making it robust to data normalization and outliers. 
	The baseline algorithm is Sinkhorn, but its convergence speed might be significantly slower for UOT than for OT. 
	In this work, we identify the cause for this deficiency, namely the lack of a global normalization of the iterates, which equivalently corresponds to a translation of the dual OT potentials.   
	Our first contribution leverages this idea to develop a provably accelerated Sinkhorn algorithm (coined ''translation invariant Sinkhorn'') for UOT, bridging the computational gap with OT.
	Our second contribution focusses on 1-D UOT and proposes a Frank-Wolfe solver applied to this translation invariant formulation. The linear oracle of each steps amounts to solving a 1-D OT problems, resulting in a linear time complexity per iteration.
	Our last contribution extends this method to the computation of UOT barycenter of 1-D measures.
	Numerical simulations showcase the convergence speed improvement brought by these three approaches.
\end{abstract}

\section{Introduction}
\label{sec-intro}

\paragraph{Optimal transport in ML.}
Optimal Transport (OT) is now used extensively to solve various ML problems. 
For probability vectors $(\al,\be) \in \RR_+^N \times\RR_+^M$, $\sum_i \al_i = \sum_j \be_j=1$, a cost matrix $\C \in \RR^{N\times M}$, it computes a coupling matrix $\pi\in\RR^{N\times M}$ solving
\begin{align*}
	\OT(\al,\be) \triangleq \inf_{\pi\geq 0,\, \pi_1=\al,\, \pi_2=\be} \dotp{\pi}{\C} = \sum_{i,j} \pi_{i,j}\C_{i,j}
\end{align*}
where $(\pi_1,\pi_2) \triangleq (\pi\ones,\pi^\top\ones)$ are the marginals of the coupling $\pi$.
%
%
The optimal transport matrix $\pi$ can be used to perform for instance domain adaptation~\cite{courty2014domain} and differentiable sorting~\cite{cuturi2019differentiable}.
If the cost is of the form $\C_{i,j}=d(x_i,x_j)^p$ where $d$ is some distance, then $\OT(\al,\be)^{1/p}$ is itself a distance between probability vectors with many favorable geometrical properties~\cite{peyre2019computational}. This distance is used for supervised learning over histograms~\cite{2015-Frogner} or unsupervised learning of generative models~\cite{WassersteinGAN}.


\paragraph{Csiszar divergences.}

The simplest formulation of UOT penalizes the discrepancy between $(\pi_1,\pi_2)$ and $(\al,\be)$ using Csiszar divergences~\cite{csiszar1967information}.
We consider an ``entropy'' function $\phi:\RR_+\rightarrow\RR_+$ which is positive, convex, lower semi-continuous and such that $\phi(1)=0$. We define $\phi^\prime_\infty\triangleq\lim_{x\rightarrow\infty} \tfrac{\phi(x)}{x}$.
Its associated Csiszar divergence reads, for $(\mu,\nu)\in\RR_+^N$
\begin{align}
	\D_\phi(\mu|\nu)\triangleq \sum_{\nu_i>0} \phi(\tfrac{\mu_i}{\nu_i})\nu_i + \phi^\prime_\infty\sum_{\nu_i=0}\mu_i.
\end{align}
A popular instance is the Kullback-Leibler divergence ($\KL$) where $\phi(x)=x\log x - x + 1$ and $\phi^\prime_\infty=+\infty$, such that $\KL(\mu|\nu)\triangleq \sum_i \log(\tfrac{\mu_i}{\nu_i})\mu_i - m(\mu) + m(\nu)$ when $(\nu_i=0)\Rightarrow(\mu_i=0)$, and $\KL(\mu|\nu)=+\infty$ otherwise.

\paragraph{Unbalanced optimal transport.}
Unbalanced optimal transport (UOT) is a generalization of OT which relaxes the constraint that $(\al,\be)$ must be probability vectors. 
Defining $m(\al)\triangleq\sum_i\al_i$ the mass of measure, we can have $m(\al)\neq m(\be)$.
%
This generalization is crucial to cope with outlier to perform robust learning~\cite{mukherjee2021outlier, balaji2020robust} and avoid to perform some a priori normalization of datasets~\cite{lee2019parallel}.
Unbalanced OT enables mass creation and destruction, which is important for instance to model growth in cell populations~\cite{schiebinger2017reconstruction}.
%
%
%
We refer to~\cite{liero2015optimal} for a thorough presentation of UOT. 
To derive efficient algorithms, following~\cite{chizat2016scaling}, we consider an entropic-regularized problem 
\begin{align}\label{eq-primal-uot-kl}
\UOT(\al,\be) \triangleq \inf_{\pi\geq 0} &\dotp{\pi}{\C} + \epsilon\KL(\pi|\al\otimes\be) \\
&+ \D_{\phi_1}(\pi_1 | \al) + \D_{\phi_2}(\pi_2 | \be)\,.\nonumber
\end{align}
Here $\KL(\pi|\al\otimes\be)$ is the Kullback-Leibler divergence between $\pi$ and $\al \otimes \be= (\al_i \be_j)_{i,j}$.
The original (unregularized) formulation of UOT~\cite{liero2015optimal} corresponds to the special case $\epsilon=0$.
A popular case uses $\D_\phi=\rho\KL$ where $\rho>0$ controls the tradeoff between mass transportation and mass creation/destruction. Balanced OT is retrieved in the limit $\rho\rightarrow +\infty$, while when $\rho \rightarrow 0$ $\UOT(\al,\be)/\rho$ converges to the Hellinger distance (no transport). 
The dual problem to~\eqref{eq-primal-uot-kl} reads $\UOT(\al,\be) = \sup_{(\f,\g)} \Ff_\epsilon(\f,\g)$, where
\begin{align}
\Ff_\epsilon(\f,\g) \triangleq &\dotp{\al}{-\phi_1^*(-\f)}
+ \dotp{\be}{-\phi_2^*(-\g)}\nonumber\\
&-\epsilon\dotp{\al\otimes\be}{e^{\tfrac{\f\oplus\g - \C}{\epsilon}} - 1}. \label{eq-dual-uot-kl}
\end{align}
Here $(\f,\g)$ are vectors in $\RR^N\times\RR^M$, $\phi^*$ is the Legendre transform of $\phi$ and we used the shorthand notations
$\f\oplus\g - \C \triangleq (\f_i+\g_j-\C_{i,j})_{i,j} \in \RR^{N \times M}, \phi_1^*(-\f) \triangleq (\phi_1^*(-\f_i))_i \in \RR^N$.
When $\epsilon=0$ the last term in~\eqref{eq-dual-uot-kl} becomes the constraint $\f\oplus\g\leq\C$.


%
%

\paragraph{Sinkhorn's algorithm and its limitation for UOT.}

Problem~\eqref{eq-primal-uot-kl} can be solved using a generalization of Sinkhorn's algorithm, which is the method of choice to solve large scale ML problem with balanced OT~\cite{CuturiSinkhorn}, since it enjoys easily parallelizable computation which streams well on GPUs~\cite{pham2020unbalanced}.
Following~\cite{chizat2016scaling, sejourne2019sinkhorn}, Sinkhorn algorithm maximizes the dual problem~\eqref{eq-dual-uot-kl} by an alternate maximization on the two variables. 
In sharp contrast with balanced OT, UOT Sinkhorn algorithm might converge slowly, even if $\epsilon$ is large.
For instance when $\D_\phi=\rho\KL$, Sinkhorn converges linearly at a rate $(1+\tfrac{\epsilon}{\rho})^{-1}$~\cite{chizat2016scaling}, which is close to $1$ and thus slow when $\epsilon\ll\rho$.
One of the main goal of this paper is to alleviate this issue by introducing a \emph{translation invariant} formulation of the dual problem together with variants of the initial Sinkhorn's iterations which enjoy better convergence rates.

%

\paragraph{Translation invariant formulations.}
The balanced OT problem corresponds to using $\phi_1(x)=\phi_2(x)=\iota_{\{1\}}(x)$ (i.e. $\phi(1)=0$ and $\phi(x)=+\infty$ otherwise), such that $\phi^*(x)=x$.
In this case $\Ff_\epsilon$ reads
\begin{align*}
\Ff_\epsilon(\f,\g) = \dotp{\al}{\f}
+ \dotp{\be}{\g}
-\epsilon\dotp{\al\otimes\be}{e^{\tfrac{\f\oplus\g - \C}{\epsilon}} - 1}. 
\end{align*}
A key property is that for any constant translation $\la\in\RR$, $\Ff_\epsilon(\f+\la,\g-\la)=\Ff_\epsilon(\f,\g)$ (because $\phi^*$ is linear), while this does not hold in general for UOT.
In particular, optimal $(\f^\star,\g^\star)$ for balanced OT are only unique up to such translation, while for $\UOT$ with strictly convex $\phi^*$ (hessian $\phi^{*\prime\prime}>0$), the dual problem has a unique pair of maximizers. 

We emphasize that this lack of translation invariance is what makes UOT slow with the following example.
Let $\f^\star$ minimize $\Ff_\epsilon$.
If one initializes Sinkhorn with $\f^\star + \tau$ for some $\tau\in\RR$, $\UOT$-$\KL$ Sinkhorn iterates read $\f_t=\f^\star + (\tfrac{\rho}{\epsilon + \rho})^{2t}\tau$.
Thus iterates are sensitive to translations, and the error $(\tfrac{\rho}{\epsilon + \rho})^{2t}\tau$ decays slowly when $\epsilon\ll\rho$.

We solve this issue by explicitly dealing with the translation parameter using an \emph{overparameterized} dual functional $\Gg_\epsilon$ and its associated invariant functional~$\Hh_\epsilon$
\begin{align}
\Gg_\epsilon(\Bf,\Bg,\la)&\triangleq \Ff_\epsilon(\Bf+\la,\, \Bg-\la),\label{eq-def-g-func}\\
\Hh_\epsilon(\Bf,\Bg)&\triangleq\sup_{\la\in\RR}\Gg_\epsilon(\Bf,\Bg,\la).\label{eq-def-h-func}
\end{align}
Note that maximizing $\Ff_\epsilon$, $\Gg_\epsilon$ or $\Hh_\epsilon$ yields the same value $\UOT(\al,\be)$ and one can switch between the maximizers using 
\begin{align}
	(f,g) = (\Bf + \la^\star(\Bf,\Bg),\Bg - \la^\star(\Bf,\Bg)) \\
 	\qwhereq
	\la^\star(\Bf,\Bg) \triangleq \argmax \Gg_\epsilon(\Bf,\Bg,\cdot). \label{eq:lambdastar}
\end{align}
By construction, one has $\Hh_\epsilon(\Bf+\la,\Bg-\la)=\Hh_\epsilon(\Bf,\Bg)$, making the functional \emph{translation invariant}.
%
%
When $\epsilon=0$ we will write $(\Gg_0,\Hh_0)$.

All our contributions design efficient UOT solvers by operating on the functionals $\Gg_\epsilon$ and $\Hh_\epsilon$ instead of $\Ff_\epsilon$.
In particular, in Section~\ref{sec-sinkhorn} we define the associated $\Gg_\epsilon$-Sinkhorn and $\Hh_\epsilon$-Sinkhorn  which performs alternate maximization on these functionals.

%

\paragraph{Solving Unregularized UOT.}
Some previous works directly adress the unregularized case $\epsilon=0$ and some specific entropy functionals.
An alternate minimization scheme is proposed in~\cite{bauer2021square} for the special case of KL divergences, but without quantitative rates of convergence.  
In the case of a quadratic divergence $\phi(x)=\rho (x-1)^2$, it is possible to compute the whole path of solutions for varying $\rho$ using LARS algorithm~\cite{chapel2021unbalanced}.
Primal-dual approaches are possible for the Wasserstein-1 case by leveraging a generalized Beckmann flow formulation~\cite{schmitzer2019framework, lee2019parallel}.
In this specific case $\epsilon=0$ and for 1-D problems, we take a different route in Section~\ref{sec-fw} by applying Frank-Wolfe algorithm on $\Hh_0$, which leads to a $O(N)$ approximate algorithm.

%
%
%

\paragraph{Solving 1-D (U)OT problems.}
1-D balanced OT is straightforward to solve because the optimal transport plan is monotonic, and is thus solvable in $O(N)$ operations once the input points supporting the distributions have been sorted (which requires $O(N \log N)$ operations). 
This is also useful to define losses in higher dimension using the sliced Wasserstein distance~\cite{bonneel2015sliced,rabin2011wasserstein}, which integrates 1-D OT problems.
%
%
%
In the specific case where $\D_\phi$ is the total variation, $\phi(x)=|x-1|$, it is possible to develop an in $O(N\log(N)^2)$ network flow solver when $\C$ represents a tree metric (so this applies in particular for 1-D problem). 
In~\cite{bonneel2019spot} an approximation is performed by considering only transport plan defining injections between the two distributions, and a $O(NM)$ algorithm is detailed. 
%
%
%
%
To the best of our knowledge, there is no general method to address 1-D UOT, and we detail in Section~\ref{sec-fw} an efficient linear time solver which applies for smooth settings.

\paragraph{Wasserstein barycenters.}
Balanced OT barycenters, as defined in~\cite{agueh-2011}, enable to perform geometric averaging of probability distribution, and finds numerous application in ML and imaging~\cite{rabin2011wasserstein,solomon2015convolutional}.
It is however challenging to solve because the support of the barycenter is apriori unknown. Its exact computation requires the solution of a multimarginal $\OT$ problems~\cite{carlier2003class, gangbo1998optimal}, it has a polynomial complexity~\cite{altschuler2021wasserstein} but does not scale to large input distributions. 
In low dimension, one can discretize the barycenter support and use standard solvers such as Frank-Wolfe methods~\cite{luise2019sinkhorn}, entropic regularization~\cite{cuturi2014fast,janati2020debiased}, interior point methods~\cite{ge2019interior} and stochastic gradient descent~\cite{li2015generative}.
These approaches could be generalized to compute UOT barycenters.
In 1-D, computing balanced OT can be done in $O(N)$ operations by operating once the support of the distributions is sorted \cite{carlier2003class, bach2019submodular}.
This approach however does not generalize to UOT, and we detail in Section~\ref{sec-barycenter} an extension of our F-W solver to compute in $O(N)$ operations an approximation of the barycenter. 
%
%
%

\paragraph{Contributions.}
%
%
Our first main contribution is the derivation in Section~\ref{sec-sinkhorn} of the $\Gg_\epsilon$-Sinkhorn algorithm (which can be applied for any divergence provided that $\phi^*$ is smooth) and $\Hh_\epsilon$-Sinkhorn algorithm (which is restricted to the KL divergence). 
%
%
%
%
We provide empirical evidence that they converge faster than the standard $\Ff_\epsilon$-Sinkhorn algorithm: when $\epsilon\leq\rho$ for $\Gg_\epsilon$, and for any $(\epsilon,\rho)$ for $\Hh_\epsilon$.
We prove that $\Hh_\epsilon$ iterates converge faster than fo $\Ff_\epsilon$ in Theorem~\ref{thm-rate-cv-h-sink}.
Section~\ref{sec-fw} details our second contribution, which is an efficient linear time approximate 1-D UOT applying Frank-Wolfe iterations to $\Hh_0$.
To the best of our knowledge, it is the first proposal which applies FW to the UOT dual, because the properties of $\Hh_0$ allow to overcome the issue of the constraint set being unbounded.
Numerical experiments show that it compares favorably against the Sinkhorn algorithm when the goal is to approximate unregularized UOT. 
In Section~\ref{sec-barycenter} we extend the Frank-Wolfe approach to compute 1-D barycenters.
All those contributions are implemented in Python, and available at \url{https://github.com/thibsej/fast\_uot}.

\section{Properties of $\Gg_\epsilon$ and $\Hh_\epsilon$}
\label{sec-trans-inv}

We first give some important properties of the optimal translation parameter, which is important to link $\Gg_\epsilon$ and $\Hh_\epsilon$. We recall that $m(\al) \triangleq \sum_i\al_i$.

\begin{prop}
	\label{prop-equiv-mass-trans}
	Assume that $\phi_1^*$, $\phi_2^*$ are smooth and strictly convex.
	Then there exists a unique maximizer $\la^\star(\Bf,\Bg)$ of $\Gg_\epsilon(\Bf,\Bg,\cdot)$. Furthermore, $(\tilde\al,\tilde\be) = \nabla \Hh_0(\Bf,\Bg)$
	satisfy $\tal = \nabla\phi_1^*(-\Bf-\la^*(\Bf,\Bg))\al$, $\tbe=\nabla\phi_2^*(-\Bg+\la^*(\Bf,\Bg))\be)$ and $m(\tal)=m(\tbe)$.
\end{prop}

\begin{proof}
	From~\cite{liero2015optimal} we have that $\lim_{x\rightarrow\infty} \phi^*(x) = +\infty$.
	Thus for any $(\Bf,\Bg)$, $\Gg_\epsilon(\Bf,\Bg,\la)\rightarrow-\infty$ when $\la\rightarrow\pm\infty$, i.e. $\Gg_\epsilon$ is coercive in $\la$.
	It means that we have compactness, and the maximum is attained in $\RR$.
	Uniqueness is given by the strict convexity of $\phi_i^*$.
	The expression of  $(\tal,\tbe)$ follows by applying the enveloppe theorem since $\phi_i^*$ are smooth. 
	Concerning the mass equality, the first order optimality condition of $\Gg_\epsilon$ in $\la$ reads $\dotp{\al}{\nabla\phi_1^*(-\Bf-\la)} = \dotp{\be}{\nabla\phi_2^*(-\Bg+\la)}$.
	Thus by definition of $(\tal,\tbe)$, this condition is rewritten as $\dotp{\tal}{1} = \dotp{\tbe}{1}$, meaning that $m(\tal) = m(\tbe)$.
\end{proof}

%
%
%

\paragraph{Closed forms for KL.}
The case $\phi_i(x)=\rho_i(x\log x -x+1)$ and $\phi_i^*(x)=\rho_i(e^{-x/\rho_i} - 1)$ (corresponding to penalties $\rho_i\KL$, where one can have $\rho_1\neq\rho_2$) enjoys simple closed form expression. 
We start with the property that if we fix $(\Bf,\Bg)$, then $\la^\star$ can be computed explicitly.
\begin{prop}
	\label{prop-kl-opt-trans}
	One has 
	\eql{\label{eq-opt-trans-kl}
		\la^\star(\Bf,\Bg) = \tfrac{\rho_1\rho_2}{\rho_1 + \rho_2} \log\Big[\frac{\dotp{\al}{e^{-\Bf / \rho_1}}}{ \dotp{\be}{e^{-\Bg / \rho_2}}}\Big].
	}
\end{prop}
\begin{proof}
	The optimality of Equation~\ref{eq-def-g-func} in $\la$ reads $\dotp{\al}{\nabla\phi_1^*(-\Bf-\la)} = \dotp{\be}{\nabla\phi_2^*(-\Bg+\la)}$,
	which for $\rho_i\KL$ is $\dotp{\al}{e^{-\tfrac{\Bf+\la}{\rho_1}}} = \dotp{\be}{e^{-\tfrac{\Bg-\la}{\rho_2}}}$. 
	Solving this equation in $\la$ yields Equation~\eqref{eq-opt-trans-kl}.
\end{proof}
Note that Equation~\eqref{eq-opt-trans-kl} can be computed in $O(N)$ time, and stabilized via a logsumexp reduction.
Equation~\eqref{eq-opt-trans-kl} is useful to rewrite $\Hh_\epsilon$ explicitly.
It yields a formulation which is new to the best of our knowledge.

\begin{prop}
	Setting $\tau_1=\tfrac{\rho_1}{\rho_1 + \rho_2}$ and $\tau_2=\tfrac{\rho_2}{\rho_1 + \rho_2}$, one has
	\begin{align}
	\Hh_\epsilon&(\Bf,\Bg) =  \rho_1 m(\al) + \rho_2 m(\be) -\epsilon\dotp{\al\otimes\be}{e^{\tfrac{\Bf\oplus\Bg - \C}{\epsilon}} - 1}\nonumber\\
	&- (\rho_1 + \rho_2)\Big(\dotp{\al}{ e^{-\Bf / \rho_1} }\Big)^{\tau_1} \Big(\dotp{\be}{ e^{-\Bg / \rho_2} }\Big)^{\tau_2}.\label{eq-dual-form-hell}
	\end{align}
	In particular when $\rho_1 = \rho_2=\rho$ and $\epsilon=0$,
	\begin{align*}
	\Hh_0(\Bf,\Bg) = \rho \Big[m(\al) +  m(\be)- 2\sqrt{   \dotp{\al}{ e^{-\Bf/\rho} }  \dotp{\be}{ e^{-\Bg/\rho} }   }\Big].
	\end{align*}
\end{prop}
%


\section{Translation invariant Sinkhorn}
\label{sec-sinkhorn}

We propose in this section two variants of the Sinkhorn algorithm based on an alternate maximization on $\Gg_\epsilon$ and $\Hh_\epsilon$.
%

%
%
\paragraph{$\Ff$-Sinkhorn (the original one).}
Sinkhorn's algorithm reads, for any initialization $\f_0$, 
\begin{align*}
\g_{t+1}(y) &= -\aprox_{\phi^*_1}(-\Smin{\al}{\epsilon}(\C(\cdot,y) - \f_{t})),\\
\f_{t+1}(x) &= -\aprox_{\phi^*_2}(-\Smin{\be}{\epsilon}(\C(x,\cdot) - \g_{t+1})),
\end{align*}
where the softmin is $\Smin{\al}{\epsilon}(\f) \triangleq -\epsilon\log\dotp{\al}{e^{-\f / \epsilon}}$, and the anisotropic prox reads 
\begin{align}
	\aprox_{\phi^*}(x)\triangleq\arg\min_{y\in\RR}\epsilon e^{\tfrac{x-y}{\epsilon}} + \phi^*(y).
\end{align}
We refer to \cite{sejourne2019sinkhorn} for more details.
For $\rho\KL$ we have $\aprox_{\phi^*}(x) = \tfrac{\rho}{\epsilon + \rho}x$.
The softmin and $\aprox_{\phi^*}$ are respectively $1$-contractive and $(1+\tfrac{\epsilon}{\rho})^{-1}$-contractive for the sup-norm $\norm{\cdot}_\infty$. 

\paragraph{$\Gg$-Sinkhorn.}
Alternate maximization on $\Gg_\epsilon$ reads
\begin{align*}
\Bg_{t+1}(y) &= -\aprox_{\phi^*_1}\big(-\Smin{\al}{\epsilon}\big(\C(\cdot,y) - \Bf_{t} - \la_{t}\big)\big), \\
\Bf_{t+1}(x) &= -\aprox_{\phi^*_2}\big(-\Smin{\be}{\epsilon}\big(\C(x,\cdot) - \Bg_{t+1} + \la_{t}\big)\big),\\
\la_{t+1} &= \la^\star(\Bf_{t+1},\Bg_{t+1}), 
\end{align*}
and the associated dual iterates are retrieved as $(\f_t,\g_t) \triangleq (\Bf_t+\la_t,  \Bg_t-\la_t)$.
%
%
For smooth $\phi^*_i$, standard results on alternate convex optimization ensure its convergence~\cite{tseng2001convergence}.
Note that the extra step to compute $\la_{t+1}$ has $O(N)$ complexity for $\KL$ (Equation~\eqref{eq-opt-trans-kl}).
For smooth $\phi^*$, computing $\la_{t+1}$ is a 1-D optimization problem whose gradient and hessian have $O(N)$ cost and converges in few iterations with a Newton method.


\paragraph{$\Hh$-Sinkhorn.}
In the following, we denote $\Psi_1(\Bg) \triangleq \argmax \Hh_\epsilon(\cdot,\Bg)$ and $\Psi_2(\Bf) \triangleq \argmax \Hh_\epsilon(\Bf,\cdot)$.
The $\Hh$-Sinkhorn's algorithm is the alternate minimization on $\Hh_\epsilon$, it thus reads 
$$
	\bar g_{t+1}=\Psi_{2}(\bar f_t), \: \bar f_{t+1}=\Psi_1(\bar g_{t+1}), 
$$
and the associated dual iterates are retrieved as $(\f_t,\g_t) \triangleq (\Bf_t+\la^\star(\Bf_t,\Bg_t),  \Bg_t-\la^\star(\Bf_t,\Bg_t))$.
Contrary to $\Gg$-Sinkhorn, $\Hh$-Sinkhorn inherits the invariance of $\Hh_\epsilon$: 
One has $\Psi_1(\Bf_t + \mu)=\Psi_1(\Bf_t) - \mu$ i.e. $\bar\g_{t+1}\rightarrow\bar\g_{t+1}-\mu$.

Computing $\Bf=\Psi_1(\Bg)$ for some fixed $\Bg$ requires to solve the equation in $\Bf$
\begin{align}\label{eq-optim-h-sink}
    e^{\Bf/\epsilon}\dotp{\be}{e^{(\Bg-\C)/\epsilon}}=\nabla\phi_1^*(-\Bf-\la^\star(\Bf,\Bg)).
\end{align}

For a generic divergence, without an explicit expression of $\la^\star$, there is a priori no closed form expression for $\f$, and one would need to re-sort to sub-iterations.
However, thanks to the closed form~\eqref{eq-opt-trans-kl} for $\KL$, the following proposition proved in the appendices shows that it can be computed in closed form.

\begin{prop}\label{prop-conv-h-sink}
	For fixed $(\Bf,\Bg)$, assuming for simplicity $\rho_1=\rho_2=\rho$, denoting $\xi \triangleq \tfrac{\epsilon}{\epsilon+2\rho}$, one has
	\begin{align*}
		\Psi_1(\Bf) = \hg + \xi \Smin{\be}{\rho}(\hg), \: \Psi_2(\Bg) = \hf + \xi  \Smin{\al}{\rho}(\hf) \\
		\text{where}\:
		\choice{
			\hg \triangleq \tfrac{\rho}{\rho+\epsilon}\Smin{\al}{\epsilon}(\C-\Bf)-\tfrac{1}{2}\tfrac{\epsilon}{\rho+\epsilon} \Smin{\al}{\rho}(\Bf),\\
			\hf \triangleq \tfrac{\rho}{\rho+\epsilon}\Smin{\be}{\epsilon}(\C-\Bg)-\tfrac{1}{2}\tfrac{\epsilon}{\rho+\epsilon} \Smin{\be}{\rho}(\Bg).		
		}
	\end{align*}
\end{prop}
%


The following theorem shows that this algorithm enjoys a better convergence rate than $\Ff$-Sinkhorn.
It involves the Hilbert pseudo-norm $\norm{\f}_\star\triangleq\inf_{t\in\RR}\norm{\f+t}_\infty$ which is relevant here due to the translation invariance of the map $\Psi_1$.
A key property that $\Hh$-Sinkhorn inherits is the contractance rate of $\Smin{\al}{\epsilon}$ for $\norm{\f}_\star$, which we write
%
	$\norm{\Smin{\al}{\epsilon}(\C-\f) - \Smin{\al}{\epsilon}(\C-\g)}_\star \leq \kappa_\epsilon(\al)\norm{\f-\g}_\star$,
%
where $\kappa_\epsilon(\al)$ denotes the contraction rate.
Local and global estimation of $\kappa_\epsilon(\al)$ are detailed respectively in~\cite{knight2008sinkhorn} and~\cite{birkhoff1957extensions, franklin1989scaling}.
The latter estimate reads $\kappa_\epsilon(\al)\leq 1 - \tfrac{2}{1 + \eta}$, where $\eta=\exp(-\tfrac{1}{2\epsilon}\max_{i,j,k,l}(\C_{j,k} + \C_{i,l} - \C_{j,l} - \C_{i,k}))$.
Note that $\eta$ depends on $\al$ via its support.

\begin{thm}\label{thm-rate-cv-h-sink}
	Write $\Bf_0$ the initialization, $(\f_t,\g_t)$ the iterates of $\Hh$-Sinkhorn after the translation $\la^\star(\Bf_t,\Bg_t)$, and $(f^\star,\g^\star)$ the optimal dual solutions for $\Ff$.
	Defining $\bar\kappa \triangleq \kappa_\epsilon(\al)\kappa_\epsilon(\be) (1+\tfrac{\epsilon}{\rho})^{-2} < 1$,  one has
	\begin{align*}
		\norm{\f_t - \f^\star}_\infty + \norm{\g_t - \g^\star}_\infty \leq 2 \bar\kappa^t \norm{\Bf_0 - \f^\star}_\star.
	\end{align*}
\end{thm}
\begin{proof}
The proof is deferred in Appendix~\ref{sec-supp-sinkhorn}
\end{proof}

The rate of $\Hh$-Sinkhorn is improved compared to its $\Ff$ counterpart by a factor $\kappa_\epsilon(\al)\kappa_\epsilon(\be)$, hence the speed-up illustrated in Figures~\ref{fig-sinkhorn-cv-rate} and~\ref{fig-sinkhorn-cv-rate-wot}.
Note that we leave a study of the overall complexity of $\Hh$-Sinkhorn for future works.
%

%
%

\begin{figure}
	\centering
	\includegraphics[width=0.4\textwidth]{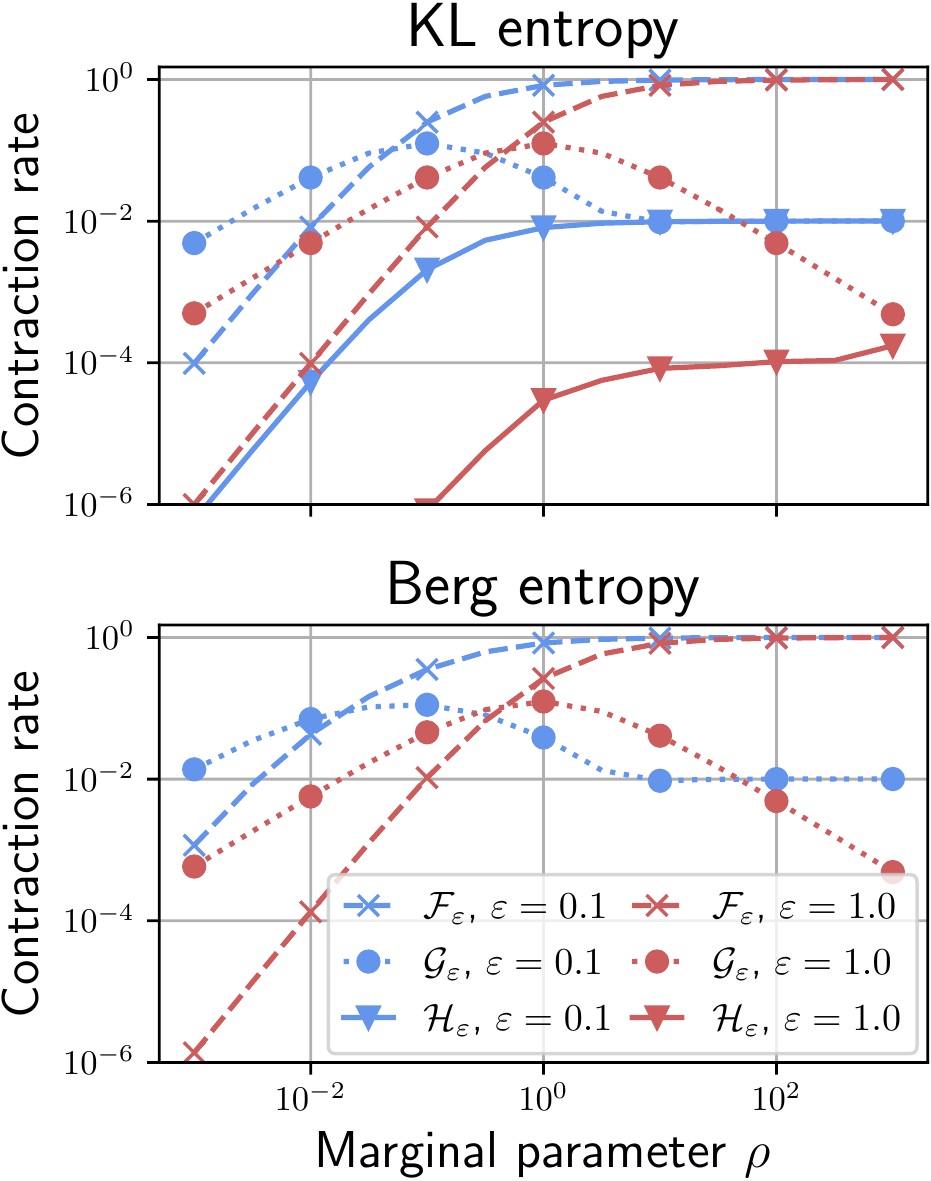}
	\caption{\textit{
		Estimation of the contraction rate of $\Ff$, $\Gg$ and $\Hh$-Sinkhorn as a function of $\rho$ for a fixed $\epsilon$. Performed on the measures of Figure~\ref{fig-iter_fw}.}
	}
	\label{fig-sinkhorn-cv-rate}
\end{figure}

\paragraph{Empirical convergence study.}
Figures~\ref{fig-sinkhorn-cv-rate} and~\ref{fig-sinkhorn-cv-rate-wot} show a numerical evaluation of the the convergence rate for $\Ff$, $\Gg$ and $\Hh$-Sinkhorn, respectively performed on synthetic 1D data (displayed Figure~\ref{fig-iter_fw}) and on the single-cell biology dataset of the WOT package\footnote{\url{https://broadinstitute.github.io/wot}}~\cite{schiebinger2017reconstruction}.
We compute the versions $(\Ff,\Gg,\Hh)$ of Sinkhorn for the $\KL$ setting.
To emphasize the generality of $\Gg$-Sinkhorn, we compute it for the Berg entropy $\phi(x) = \rho(x - 1 - \log x)$ (and $\phi^*(x) = -\rho\log(1-\tfrac{x}{\rho})$)
We observe empirically that all versions converge linearly to fixed points $(\f^\star,\g^\star)$.
Thus we focus on estimating the convergence rate $\kappa$.
Given iterates $\f_t$, it is estimated as $\kappa=e^c$ where $c$ is the median over $t$ of $\log\norm{\f_{t+1}-\f^\star}_\infty - \log\norm{\f_{t}-\f^\star}_\infty$.
We report those estimates as curves where $\rho$ varies while $\epsilon$ is fixed.

\begin{figure}
	\centering
	\includegraphics[width=0.35\textwidth]{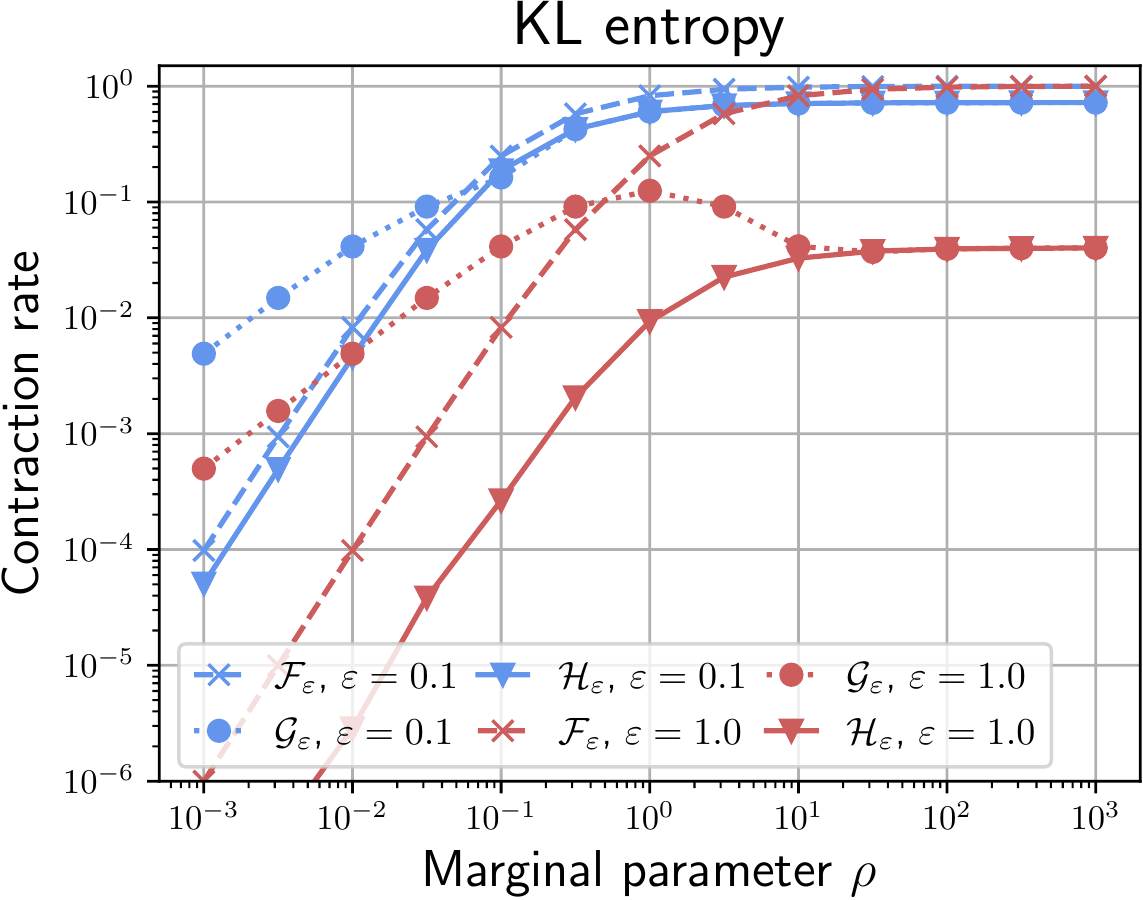}
	\caption{\textit{
			Estimation of the contraction rate of $\Ff$, $\Gg$ and $\Hh$-Sinkhorn as a function of $\rho$ for a fixed $\epsilon$. Performed on WOT single-cell data.}
	}
	\label{fig-sinkhorn-cv-rate-wot}
\end{figure}

We observe that $\Hh$-Sinkhorn outperforms both $\Gg$ and $\Ff$-Sinkhorn, and its convergence curve appears as a translation of the $\Ff$-Sinkhorn curve (of approximately $\log(\kappa_\epsilon(\al)\kappa_\epsilon(\be))<0$).
Note also that the overall complexity of $\Hh$-Sinkhorn remains $O(N^2)$ because $\Smin{}{\rho}$ translations cost $O(N)$.
Concerning $\Gg$-Sinkhorn, it outperforms $\Ff$-Sinkhorn in its slow regime $\epsilon\leq\rho$, but is slower when $\rho\leq\epsilon$.
This behaviour is consistent with $\KL$ and Berg entropies.
Thus the criteria whether $\epsilon\leq\rho$ or not seems a correct rule of thumb to decide when $\Ff$ or $\Gg$-Sinkhorn is preferable.

\paragraph{Extensions.}
%
It is also possible to accelerate the convergence of Sinkhorn using Anderson extrapolation~\cite{anderson1965iterative}, see Appendix for details.
%


\section{Frank-Wolfe solver in 1-D}
\label{sec-fw}

In this section, we derive an efficient 1-D solver using Frank-Wolfe's algorithm in the unregularized setting $\epsilon=0$.
Frank-Wolfe or conditional gradient method~\cite{frank1956algorithm} minimizes a smooth convex function on a compact convex set by linearizing the function at each iteration. 
It is tempting to apply F-W's algorithm to solve the UOT problem since the resulting linearized problem is a balanced OT problem, which itself can be solved efficiently in 1-D.
One cannot however directly apply F-W to $\Ff_0$ because the associated constraint set $f \oplus g \leq \C$ (i.e. $\f_i+\g_j\leq\C_{i,j}$) is a priori unbounded, since it is left unchanged by the translation $(f+\la,g-\la)$. 
We thus propose to rather apply it on the translation invariant functional $\Hh_0$, which results in an efficient numerical scheme, which we now detail.

\paragraph{F-W for UOT.}

We apply F-W to the problem $\sup_{\Bf\oplus\Bg\leq\C} \Hh_0(\Bf,\Bg)$.
While the constraint set is a priori unbounded, we will show that the iterates remain well defined nevertheless. 
The iterations of F-W read 
$$
    (\Bf_{t+1},\Bg_{t+1}) = (1-\ga_t) (\Bf_t,\Bg_t) + \ga_t (r_{t},s_{t})
$$
for some step size $\ga_t>0$ where $(f_{t},g_{t})$ are solutions of a Linear Minmization Oracle (LMO), 
$$
    (r_{t},s_{t}) \in \uargmin{r \oplus s \leq C} \dotp{ (r,s) }{ \nabla \Hh_0(\Bf_t,\Bg_t) }.
$$
Thanks to Proposition~\ref{prop-equiv-mass-trans}, this LMO thus reads
\begin{align}
    &(r_{t},s_{t}) \in \uargmin{r \oplus s \leq C}   \dotp{r}{\tilde\al_t} + \dotp{s}{\tilde\be_t}, \\
    \qwhereq
    &\choice{
    \tal_t \triangleq \nabla\phi_1^*(-\Bf_t-\la^\star(\Bf_t,\Bg_t))\al, \\
    \tbe_t \triangleq \nabla\phi_2^*(-\Bg_t+\la^\star(\Bf_t,\Bg_t))\be.
    }\label{eq:fw-updated-histo}
\end{align}
It is thus the solution of a balanced OT problem between two histograms with equal masses. Hence the iteration of this F-W are well-defined.
Recall that $\la^\star$ is computable in closed form for $\KL$ (Proposition~\ref{prop-kl-opt-trans}) or via a Newton scheme for smooth $\phi^*$.

Note that this approach holds for any measures defined on any space.
Thus one could use any algorithm such as network simplex or Sinkhorn.
While the computational gain of this approach is not clear in general, we propose to focus on the setting of 1-D data, where the LMO is particularly fast to solve.

	\begin{algorithm}[t]
		\caption{-- \textbf{SolveOT1D($x$, $\al$, $y$, $\be$, $\C$)} }\label{alg-ot-lmo}
		\small{
			\textbf{Input:} measures $(x,\al,N)$ and $(y,\be,M)$, cost $\C$\\
			\textbf{Output:} primal-dual solutions $(\pi, \f,\g)$\\
			\begin{algorithmic}[1]
				\STATE Set $\pi,\, \f,\, \g \leftarrow 0,\, 0,\, 0$
				\STATE Set $\g_1,\, a,\, b,\, i,\, j \leftarrow \C(x_1,y_1),\, \al_1,\, \be_1,\, 1,\, 1$
				\WHILE{$i<N$ or $j<M$}
				\IF{($a\geq b$ and $i<N$) or ($j=M$)}
				\STATE $\pi_{i,j},\, b\leftarrow a,\, b - a$
				\STATE $i\leftarrow i + 1$
				\STATE $\f_i,\, a\leftarrow \C(x_i, y_j) - \g_j,\, \al_i$
				\ELSIF{($a>b$ and $j<M$) or ($i=N$)}
				\STATE $\pi_{i,j},\, a\leftarrow b,\, a - b$
				\STATE $j\leftarrow j + 1$
				\STATE $\g_j,\, b\leftarrow \C(x_i, y_j) - \f_i,\, \be_j$
				\ENDIF
				\ENDWHILE
				\STATE Return $(\pi,\f,\g)$.
			\end{algorithmic}
		}
	\end{algorithm}

\paragraph{The 1-D case.}

Algorithm~\ref{alg-ot-lmo} details a fast and exact $O(N+M)$ time solver for 1-D optimal transport when $C_{i,j}=|x_i-y_j|^p$ ($p\geq 1$) in 1-D and the points are already sorted. It can thus be used to compute the LMO for 1-D UOT problems.

Algorithm~\ref{algo-fw} details the resulting F-W algorithm for 1-D UOT.
It uses either the standard step size $\gamma_t=\tfrac{2}{2+t}$ or a line search optimizing 
$$
    \ga \in [0,1] \mapsto \Hh_0( (1-\ga) \Bf_t + \gamma r_t, (1-\ga) \Bg_t + \gamma s_t ).
$$
For a KL divergence, the computation of $(\tal_t,\tbe_t)$ is in closed form and require $O(N+M)$ operations, while for other divergences, it can be obtained  using a few iterations of a Newton solver. The induced cost is in any case comparable with the one of the SolveOT1D sub-routine.  
We also test the Pairwise FW (PFW) variant~\cite{lacoste2015global}, which we detail in the Appendix. This variant requires to store all the iterates of the algorithm (thus being memory intensive to reach high precision), but ensures a linear convergence under looser additional conditions than FW.

\begin{figure*}[h!]
	\centering
	\begin{tabular}{c@{}c@{}c}
		{\includegraphics[width=0.3\linewidth]{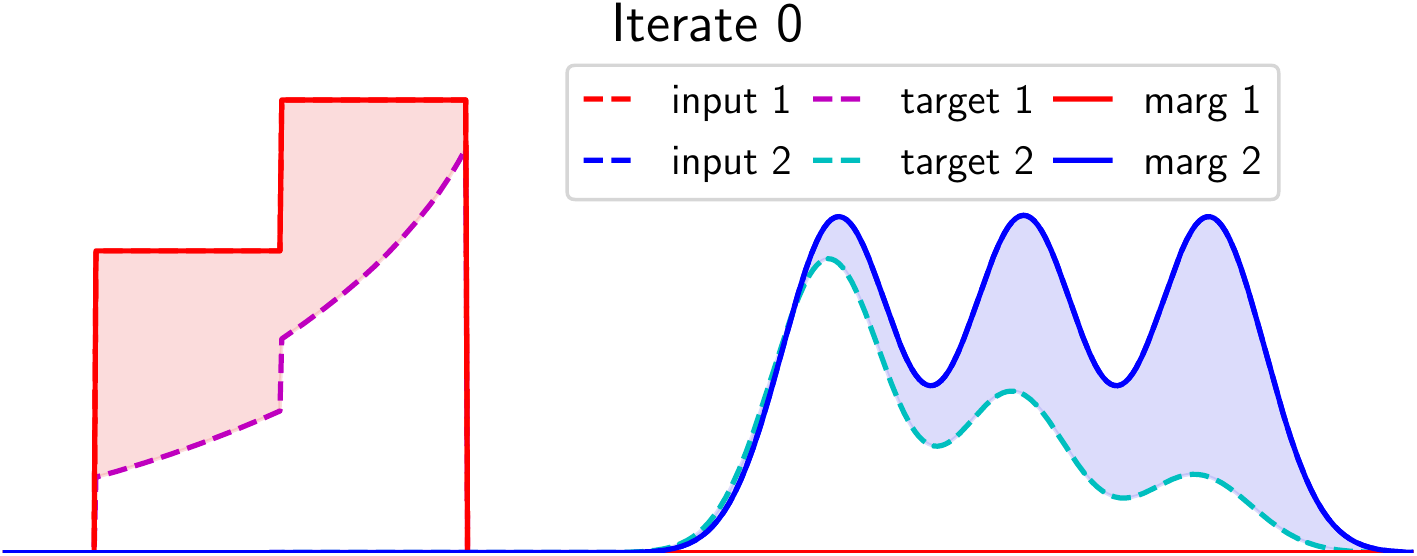}} &
		{\includegraphics[width=0.3\linewidth]{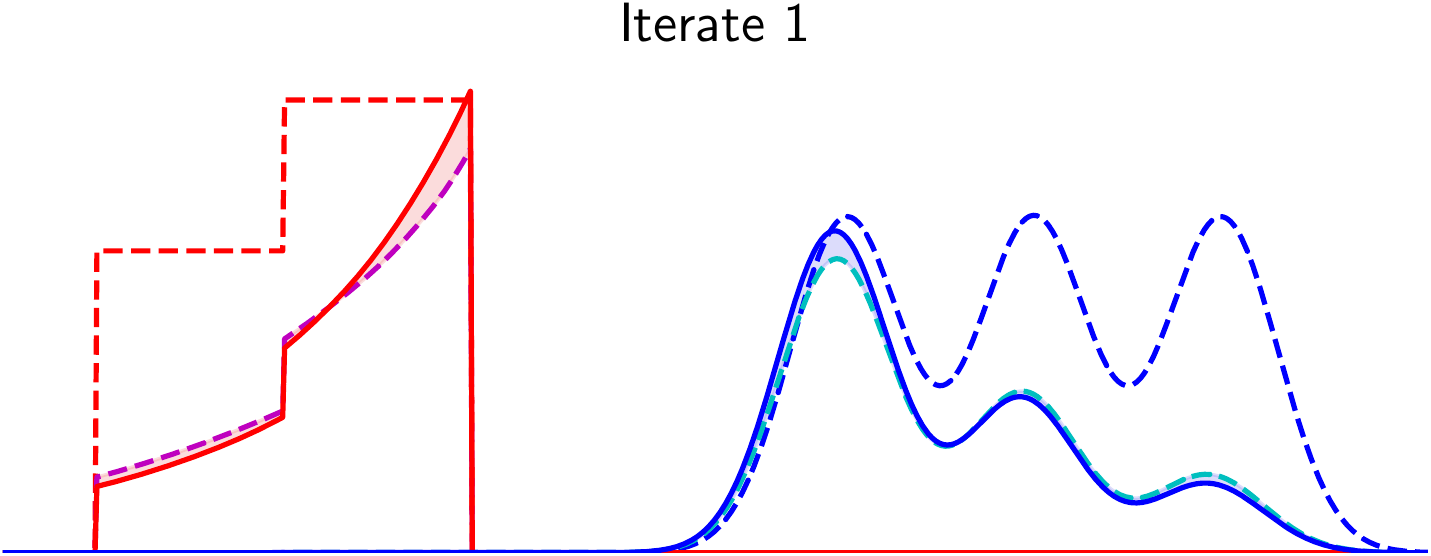}} &
		{\includegraphics[width=0.3\linewidth]{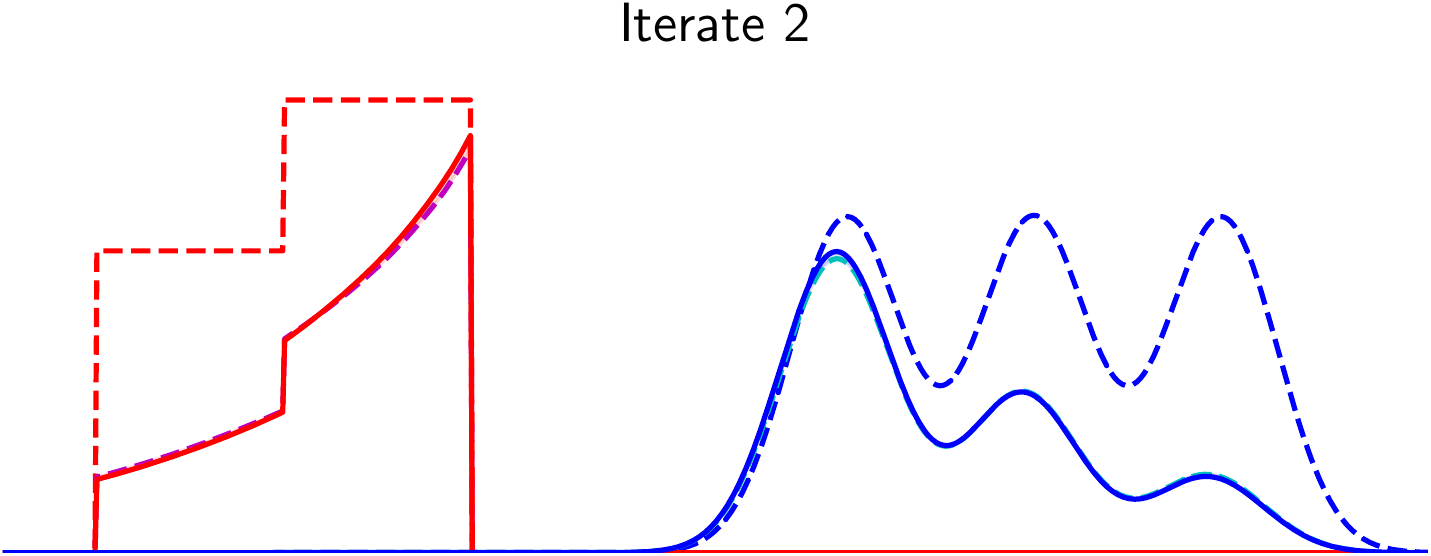}}
	\end{tabular}
	\caption{
		\textit{
		Evolution of the plan's marginals at the first iterations. The inputs $(\al,\be)$ are the dashed lines, the optimal marginals $(\pi_1^\star,\pi_2^\star)$ are dotted in cyan/magenta. The initialization is $(\f_0,\g_0)=(0,0)$, thus $(\tal_0,\tbe_0)=(\al,\be)$. The filled area is the error between $(\tal_t,\tbe_t)$ and $(\pi_1^\star,\pi_2^\star)$. 
	}
	}
	\label{fig-iter_fw}
\end{figure*}

We provide an illustration of Algorithm~\ref{algo-fw} on Figure~\ref{fig-iter_fw} to illustrate the optimization from the primal point of view.
At optimality, in the $\KL$ setting, one has $\pi_1^\star=e^{-\f^\star / \rho}\al$ and $\pi_2^\star=e^{-\g^\star / \rho}\be$.
Thus we can estimate suboptimal marginals as $\pi_{1,t}=e^{-\f_t / \rho}\al$ and $\pi_{2,t}=e^{-\g_t / \rho}\be$, where $(\f_t,\g_t)$ are the FW iterates.
We observe that the term $e^{-\f / \rho}$ acts as a normalization on the marginals.
We also observe that on these examples, the marginals are close to $(\pi_1^\star,\pi_2^\star)$ after only 2 iterations (Iteration zero is the initialization $(\pi_{1,0},\pi_{2,0})=(\al,\be)$).

\begin{figure}[H]
	\centering
	\includegraphics[width=0.35\textwidth]{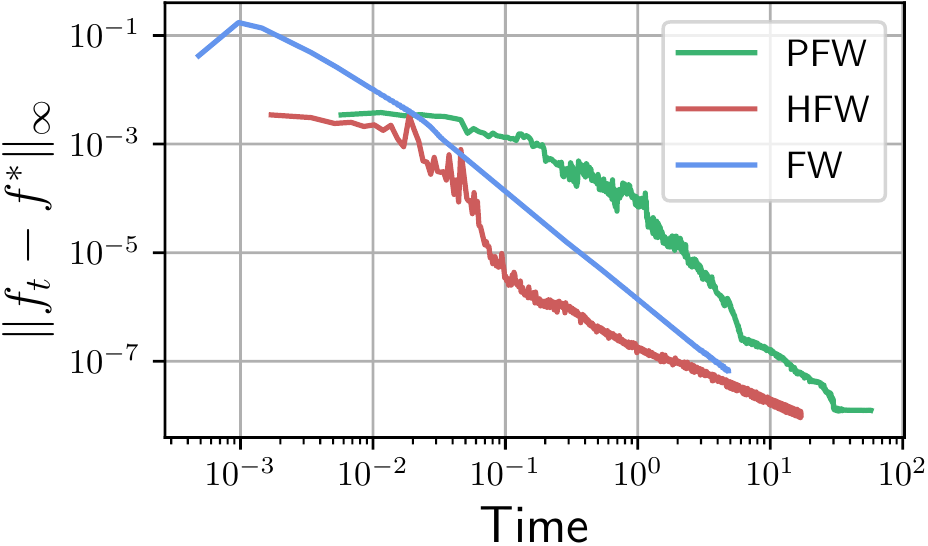}
	\caption{\textit{
		Comparison of FW with and without line-search and PFW during $10.000$ iterations. The computation time per iteration is averaged and reported. We display the error $\norm{\f_t-\f^\star}_\infty$.}
	}
	\label{fig-comparison-fw-ver}
\end{figure}

%
%
%
%
%
%
%
%
%

	\begin{algorithm}[t]
		\caption{-- \textbf{SolveUOT($x$, $\al$, $y$, $\be$, $\C$, $\rho_1$, $\rho_2$)} }\label{algo-fw}
		\small{
			\textbf{Input:} sorted $(x,y)$, histograms $(\al,\be)$, cost $\C$.  \\
			\textbf{Output:}~dual potentials $(\f_t,\g_t)$ \newline \vspace{-4mm}
			\begin{algorithmic}[1]
				\STATE Initialize  $(\Bf_0,\Bg_0)$, $t=0$
				\WHILE{$(\Bf_t,\Bg_t)$ has not converged}
				\STATE Compute $(\tal_t,\tbe_t)$ using~\eqref{eq:fw-updated-histo}.
				\STATE $(r_t,s_t) \leftarrow$ SolveOT($x$, $\tal_t$, $y$, $\tbe_t$, $\C$)
				\STATE $\gamma_t = \text{LineSearch} \Bf_t, \Bg_t, r_t, s_t) \text{ or }
				    \gamma_t=\tfrac{2}{2+t}$
				\STATE  $(\Bf_{t+1},\Bg_{t+1}) = (1-\ga_t) (\Bf_t,\Bg_t) + \ga_t (r_{t},s_{t})$, 
				    $t \leftarrow t+1$
				\ENDWHILE
				\STATE Return $(\f_t,\g_t) \triangleq (\Bf_t+\la^\star(\Bf_t,\Bg_t),  \Bg_t-\la^\star(\Bf_t,\Bg_t))$.
			\end{algorithmic}
		}
	\end{algorithm}
%

We showcase a comparison of FW (with and without linesearch on $\Hh_0$) and PFW on Figure~\ref{fig-comparison-fw-ver}.
We solve the $\UOT$ problem for $\rho=10^{-1}$ between the 1-D measures displayed Figure~\ref{fig-iter_fw}, each one having $5.000$ samples.
We run $10.000$ iterations and compare the potential $\f_t$ after we precomputed $\f^\star$.
We report the computation time to see whether or not the gain of a line-search is worth the extra computation time.
In this example we observe that FW with lineseach outperforms.
We also observe that the three variants have linear convergence.

\paragraph{Comparison of performance.}
We now compare our implementation with the Sinkhorn algorithm, which is the reference algorithm, and is especially tailored for GPU architectures.
%
%
We consider two histograms of size $N=M=200$, set $\rho=1$, compute the optimal potentials $(\f^\star,\g^\star)$ with CVXPY~\cite{diamond2016cvxpy},  run $5000$ iterations of FW without line-search and  $\Hh$-Sinkhorn and report $\norm{\f_t-\f^\star}_\infty$.
We consider Sinkhorn only on GPU and use stabilized soft-max operations since we wish to approximate the unregularized problem where $\epsilon$ should be small.
We also perform log-stable updates in FW when we compute the translation $\la^\star$.
%
%
We display the result on Figure~\ref{fig-sinkhorn-vs-fw}, where the horizontal axis refers to computation time.
Note that even for problems of such a small size, the $O(N^2)$ cost per iteration of Sinkhorn dominates the $O(N)$ cost of FW. FW provides a better estimate of $\f^\star$ during all the iterations for a wide range of $\epsilon$.
%
%
Additional results in Appendix show that that this behaviour is similar for other values of $\rho$.

\section{Barycenters}
\label{sec-barycenter}

\paragraph{UOT barycenters.}

To be able to derive an efficient F-W procedure for the computation of barycenters, we consider in this section asymmetric relaxations, where $\pi_1$ is penalized with $\rho\KL$ and we impose $\pi_2=\be$  (where $\be$ represents here the barycenter).
To emphasize the role of the positions of the support of the histograms, we denote the resulting UOT value as 
\begin{align*}
\UW( (\al,x), (\be,y) ) \triangleq \umin{\pi\geq 0, \pi_2 = \be} \dotp{\pi}{C} +\D_\phi(\pi_1 | \al),
\end{align*}
where the cost is $\C_{i,j} = c(x,y_j)$ for some ground cost~$c$.

%
%
We consider in this section the barycenter problem between $K$ measures $(\al_1,\ldots,\al_K)\in\RR_+^{N_1}\times\ldots\times\RR_+^{N_K}$, each measure being supported on a set of points $x_k = (x_{(k, 1)}, \ldots, x_{(k, N_k)})$.
It reads
\begin{align}\label{eq-bar}
\umin{\be,y} \sum_{k=1}^K \om_k \UW((\al_k,x_k),(\be,y)),
\end{align}
where $\om_k \geq 0$ and $\sum_k \om_k=1$.

\paragraph{Multi-marginal formulation.}

The main difficulty in computing such barycenters is that the support $y$ is unknown, and the problem is non-convex with respect to $y$.
Proposition~\ref{prop:barycentermulti} below states that this barycenter can be equivalently computed by solving the following convex unbalanced multi-marginal problem
\begin{align}\label{eq-multimarg}
    \umin{\ga\geq 0} \dotp{\ga}{\Cc} + \sum_{k=1}^K \om_k \D_{\phi}(\ga_k | \al_k),
\end{align}
where $\ga_k$ is the $k^{\text{th}}$ marginal of the tensor $\ga \in \RR^{N_1 \times \ldots \times N_K}$
obtained by summing over all indices but $k$. 
%
%
The cost of this multi-marginal problem is 
\begin{align}
	\Cc_{(i_1,\ldots,i_K)} &\triangleq \umin{b}\sum_k \om_k c(x_{(k,i_k)},b). \label{EqCostMultiMarginal}
\end{align}
For instance, when $c(x,y) = \norm{x-y}^2$, then up to additive constants, one has 
$\Cc_{(i_1,\ldots,i_K)} = - 2 \sum_{k \neq \ell} \dotp{x_{(k, i_k)}}{x_{(\ell, i_\ell)}}$.
We consider this setting in our experiments.
In the following, we make use of the barycentric map
\begin{equation*}
    B_\om(z_1,\ldots,z_k) \triangleq \uargmin{b} \sum_k \om_k c(z_k,b).
\end{equation*}
For instance, one has $B_\om(z_1,\ldots,z_k) = \sum_k \om_k z_k$ for $c(x,y) = \norm{x-y}^2$.

\begin{prop}\label{prop:barycentermulti}
	Problems~\eqref{eq-bar} and~\eqref{eq-multimarg} have equal value. 
	Furthermore, for any optimal multimarginal plan $\ga^\star$, with support $I = \{i : \ga_i^\star \neq 0\}$, an optimal barycenter for problem~\eqref{eq-bar} is supported on the set of points $y = ( B_\om(x_{(1, i_1)},\ldots,x_{(K, i_K)}) )_{i \in I}$ with associated weights $\be = (\ga^\star_i)_{i \in I} \in \RR^{|I|}$.
	%
\end{prop}
\begin{proof}
    We can parameterize Program~\eqref{eq-bar} with variables $(\ga_1,\ldots,\ga_K)$ such that $\UW((\al_k,x_k),(\be,y))$ amounts to solve $\OT(\ga_k,\be)$ for an optimal choice of $(\ga_1,\ldots,\ga_K)$. Thus $\be$ is the solution of a balanced barycenter problem with inputs $(\ga_1,\ldots,\ga_K)$. We have from~\cite{agueh-2011} the equivalence with the multimarginal problem, hence the result.
\end{proof}

While convex, Problem~\eqref{eq-multimarg} is in general intractable because its size grows exponentially with $K$.
A noticeable exception, which we detail below, is the balanced case in 1-D, when $\D_{\phi} = \iota_{\{=\}}$ imposes $\ga_k = \al_k$, in which case it can be solved in linear time $O(\sum_k N_k)$, with an extension of the 1-D OT algorithm detailed in the previous section. This is the core of our F-W method to solve the initial barycenter problem.

%
%

\begin{figure}
	\centering
	\includegraphics[width=0.35\textwidth]{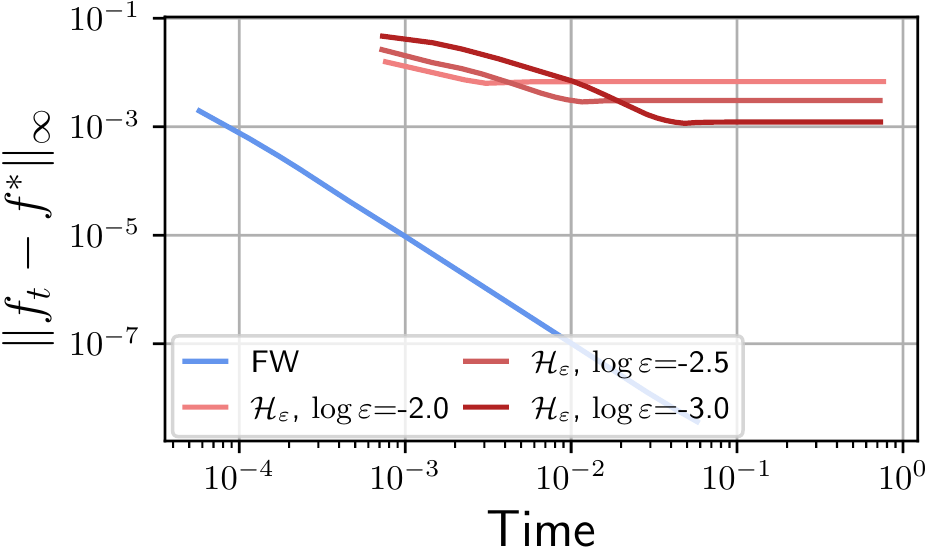}
	\caption{\textit{
		Comparison of $\norm{\f_t-\f^\star}_\infty$ depending on time for $5000$ iterations of Sinkhorn or FW without linesearch. The value of $\log_{10}(\epsilon)$ is reported in the legend.} 
	}
	\label{fig-sinkhorn-vs-fw}
\end{figure}

\begin{figure*}
	\centering
	\begin{tabular}{c@{}c@{}c@{}c@{}c}
		{\includegraphics[width=0.3\linewidth]{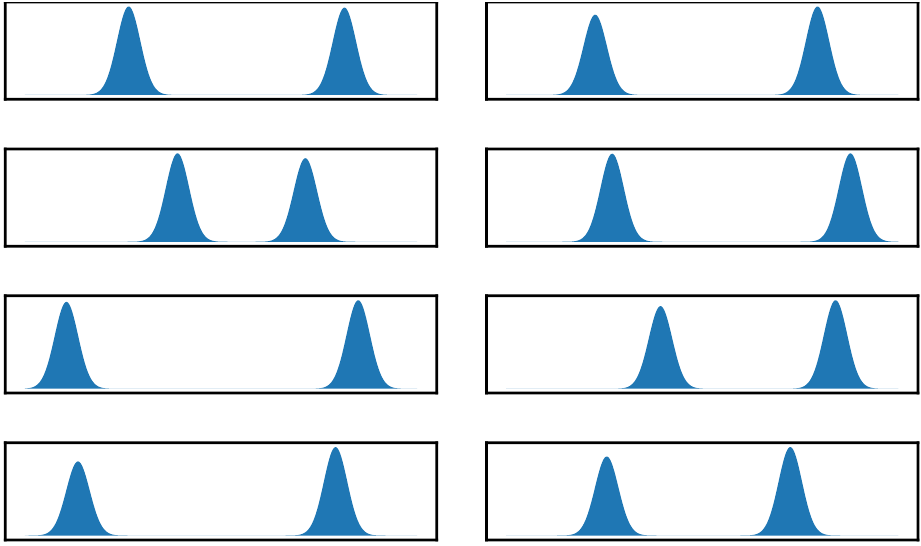}} & $\quad$ &
		{\includegraphics[width=0.3\linewidth]{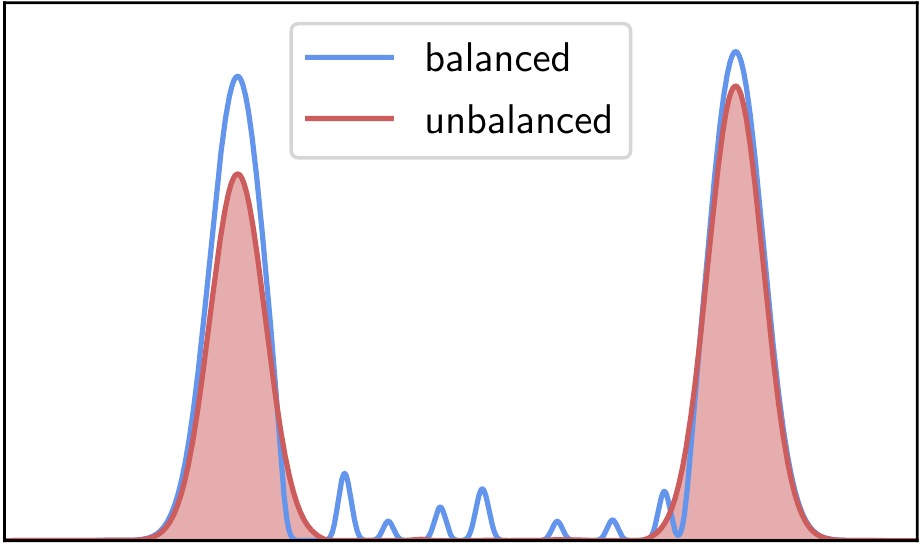}} & $\quad$ &
		{\includegraphics[width=0.3\linewidth]{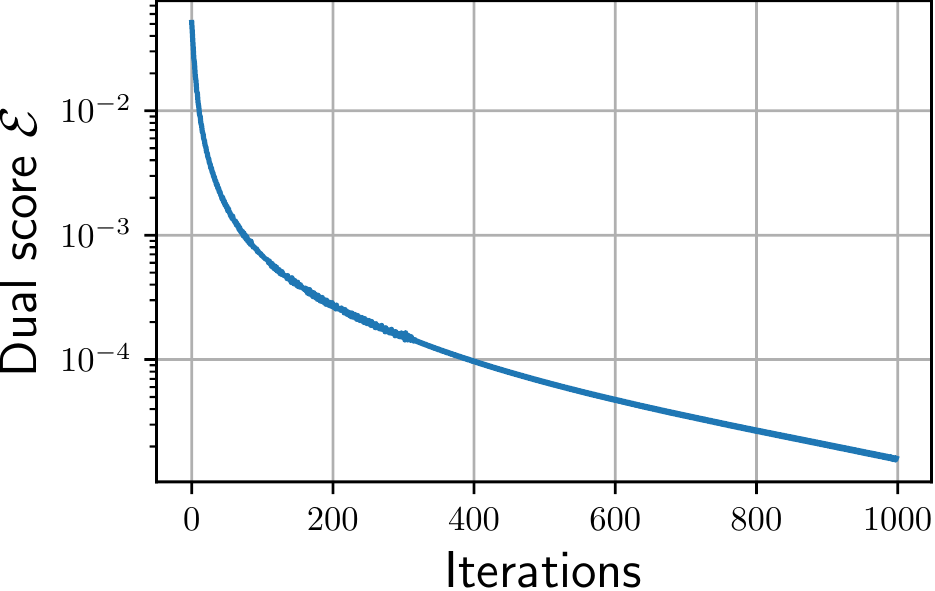}}
	\end{tabular}
	\caption{\textit{
			Left: plot of 8 random mixtures supported in $[0,1]$ used to compute their isobarycenter. Center: Barycenter for $\OT$ and $\UOT$. Right: Value of dual objective $\Hh((\Bf^\star)_k) - \Hh((\Bf_t)_k)$.}  
	}
	\label{fig-bar-fw}
\end{figure*}

\paragraph{Solving the 1-D balanced multimarginal problem.}

Algorithm \ref{algo-multimarg-ot} solves the multi-marginal problem~\eqref{eq-multimarg} in the balanced case $\D_{\phi} = \iota_{\{=\}}$.
It is valid in the 1-D case, when $c(x,y)=|x-y|^p$ for $p \geq 1$, and more generally when $\Cc$ satisfies some submodularity condition as defined in \cite{bach2019submodular, carlier2003class}.
We show the correctness of this algorithm in the Appendix.
Note that Algorithm~\ref{algo-multimarg-ot} differs from~\cite{cohen2021sliced, bach2019submodular} which solves 1D multimarginal OT by computing $\ell_2$ norms of inverse cumulative distribution functions in the barycenter setting.
While both approaches can be used to backpropagate through the multimarginal loss, Algorithm~\ref{algo-multimarg-ot} holds for more general costs, and allows to compute an explicit plan and thus the barycenter.

	\begin{algorithm}[t]
		\caption{-- \textbf{SolveMOT($(x_k)_k$, $(\al_k)_l$, $\om$, $\Cc$)} }\label{algo-multimarg-ot}
		\small{
			\textbf{Input:} $K$ measures $(x_k,\al_k,N_k)$, $K$ weights $\om_k$ and multimarginal cost $\Cc$\\
			\textbf{Output:}~primal-dual solutions $(\ga,\, \{\f_k\}_k)$\\
			\begin{algorithmic}[1]
				\STATE Set $\ga,\, \f_{k} \leftarrow 0,\, 0$
				\STATE Set $ \{a_k\}_k ,\, \{i_k\}_k \, \leftarrow \,\{\al_{(k, 1)}\}_k ,\, \{1\}_k$
				\STATE Set $\f_{(1, 1)} \leftarrow \Cc(x_{(1, 1)},\dots,x_{(K, 1)})$
				\WHILE{$\exists k,\, i_k < N_k$}
				\STATE $p\leftarrow\arg\min\{k\mapsto a_k \,\, \textrm{s.t.}\,\, i_k<N_k\}$
				\STATE $\ga_{(i_1,\ldots,i_K)},\,\{a_k\}_k\leftarrow a_p,\,   \{a_k - a_p\}_k$
				\STATE  $i_p\leftarrow  i_p + 1$
				\STATE $\f_{(p, i_p)}  \leftarrow  \Cc(x_{(1, i_1)},\dots,x_{(K, i_K)}) - \sum_{k\neq p}\f_{(k,i_k)}$
				\STATE $ a_p  \leftarrow \al_{(p, i_p)}$
				\ENDWHILE
				\STATE Return $(\ga,\, \{\f_k\}_k)$.
			\end{algorithmic}
		}
	\end{algorithm}

\paragraph{Translation invariant multi-marginal.}
The dual of the multi-marginal problem~\eqref{eq-multimarg} reads, for $\f = (\f_1,\ldots,\f_K)$,
\begin{align}
    \umax{\f_1\oplus\ldots\oplus\f_K\leq\Cc} \Ff(\f) \triangleq \sum_{k=1}^K \dotp{\al_k}{-\om_k\phi^*(-\tfrac{\f_k}{\om_k}) }.
	 \label{eq-dual-multimarg}
\end{align}
Similarly to Section~\ref{sec-trans-inv}, we define a translation invariant functional for $\Bf = (\Bf_1,\ldots,\Bf_K)$
%
\begin{align*}
	\Hh(\Bf) 
	\triangleq
	\sup_{\sum_k \la_k=0}\Ff(\Bf_1+\la_1,\ldots,\Bf_K+\la_K).
\end{align*}
The following proposition generalizes Proposition~\ref{prop-equiv-mass-trans}.

\begin{prop}\label{prop-optim-trans-multimarg}
	Assume that $\phi^*$ is smooth and strictly convex. Then there exists a unique 
	$\la^\star(\Bf) = (\la_1,\ldots,\la_K)$ s.t. $\sum_k \la_k=0$ in the definition of $\Hh$.
	Then $\nabla \Hh(\Bf) = \tal$ where 
	$\tal_k \triangleq \nabla\phi^*(-\Bf_k-\la_k)\al_k$ is such that for any $(i,j)$ one has $m(\tal_i)=m(\tal_j)$.
\end{prop}

\begin{proof}
	For this proof we reparameterize $(\la_1,\ldots,\la_K)$ as $(\la_1-\La,\ldots,\la_K-\La)$ where $\La =\tfrac{1}{K}\sum_k \la_k$, such that the constraint $\sum_k \la_k=0$ can be dropped.
	If any $\la_i\rightarrow-\infty$ then $\Ff(\ldots)\rightarrow-\infty$.
	If any $\la_i\rightarrow+\infty$ then $\La\rightarrow+\infty$ and again $\Dd(\ldots)\rightarrow-\infty$.
	Thus we are in a coercive setting, and there is existence of a minimizers.
	Uniqueness is given by the strict convexity of $\phi^*$.
	Finally, the first order optimality condition w.r.t. $\la_i$ reads $\dotp{\al_i}{\nabla\phi^*(-\f_i -\la_i + \La)} = \sum_k \dotp{\al_k}{\nabla\phi^*(-\f_k - \la_k + \La)}$.
	The r.h.s. term is the same for any $i$, thus for any $(i,j)$ one has $\dotp{\al_i}{\nabla\phi^*(-\f_i -\la_i + \La)} = \dotp{\al_j}{\nabla\phi^*(-\f_j -\la_j + \La)}$, which reads $m(\tal_i)=m(\tal_j)$.
\end{proof}

The following proposition shows that for $\rho \KL$ divergences, one can compute the optimal translation in closed form.
%

\begin{prop}[$\KL$ setting]
	%
	When $\D_{\phi} = \rho \KL$, 
	then, denoting $q_k \triangleq \log\dotp{\al_k}{e^{-\f_k / (\om_k \rho)}}$, 
	\eql{\label{eq-optim-const-multimarg}
		\la^\star(\Bf)_i = \om_i \rho q_i - \frac{\om_i}{\sum_k \om_k} \sum_{k=1}^K \om_k \rho q_k.
	}
\end{prop}
Note that Equation~\eqref{eq-optim-const-multimarg} can be computed in $O(\sum_k N_k)$ time, and could be used in the multi-marginal Sinkhorn algorithm to potentially improve its convergence.

\paragraph{Unbalanced multi-marginal F-W.}

Similarly to Sections~\ref{sec-trans-inv} and~\ref{sec-fw}, we propose to optimize the multimarginal problem~\eqref{eq-multimarg} via F-W applied to the problem 
\begin{equation*}
    \umax{\Bf_1\oplus\ldots\oplus\Bf_K\leq\Cc} \Hh(\Bf).
\end{equation*}
Each F-W step has the form $\Bf^{(t+1)} = (1-\tau_t) \Bf^{(t)} + \tau_t r$ where, thanks to Proposition~\ref{prop-optim-trans-multimarg}, $r=(r_1,\ldots,r_K)$ solves the following LMO
\eq{
	\umax{r_1 \oplus\ldots\oplus r_K\leq\Cc} \sum_{k=1}^K \dotp{\tilde\al_k^{(t)}}{r_k}
	\text{ with }
	\tilde\al_k^{(t)} \triangleq \nabla\phi_k^*(-\Bf_k^{(t)} - \la_k).
}
Proposition~\ref{prop-optim-trans-multimarg} guarantees that all the $\tilde\al_k^{(t)}$, have the same mass, so that the F-W iterations are well defined. 
In 1-D, this LMO can thus be solved in linear time using the function SolveMOT($(x_k)_k, (\tilde\al_k^{(t)})_k, \om, \Cc$).


\paragraph{Numerical experiments.}
Figure~\ref{fig-bar-fw} displays an example of computation of the balanced OT (corresponding to using $\rho=+\infty$) and UW barycenters of Gaussian mixtures.
We consider the isobarycenter, $\om_k=\tfrac{1}{K}$.
The input measures are $K=8$ Gaussian mixtures $a\cdot\Nn(\mu_1,\sigma) + b\cdot\Nn(\mu_2,\sigma)$ where $\sigma=0.03$, $\mu_1\sim\Uu([0.1,\,0.4])$, $\mu_2\sim\Uu([0.6,\,0.9])$ and $(a,b)\sim\Uu([0.8,\,0.1])$.
The input measures and the barycenter are not densities, but $5.000$ samples smoothed on Figure~\ref{fig-bar-fw} using Gaussian kernel density estimation.
One can observe that both $\OT$ and $\UW$ retrieve two modes.
Note however that $\OT$ barycenter displays multiple undesirable minor modes between the two main modes.
This highlights the ability of the unbalanced barycenter to cope with mass variations in the modes of the input distribution, which create undesirable artifact in the balanced barycenter.


\section{Conclusion}

We presented in this paper a translation invariant reformulation of $\UOT$ problems.
While conceptually simple, this modification allows to make Sinkhorn's iterations as fast  in the unbalanced as in the balanced case.
This also allows to operate F-W steps, which turns out to be very efficient for 1-D problems.

\section*{Acknowledgements}

The work of Gabriel Peyr\'e was supported by the French government under management of Agence Nationale de la Recherche as part of the ``Investissements d'avenir'' program, reference ANR19-P3IA-0001 (PRAIRIE 3IA Institute) and by the European Research Council (ERC project NORIA). 

\bibliography{biblio.bib}
%
%
\onecolumn
\appendix

\section{Appendix of Section 2 - Translation Invariance}

\subsection{Detailed proof of Proposition 2}

\begin{proof}
	Recall the optimality condition in the $\KL$ setting which reads
	\begin{align*}
		&\dotp{\al}{e^{-\frac{\Bf +\la^\star}{\rho_1}}} = \dotp{\be}{e^{-\frac{\Bg -\la^\star}{\rho_2}}},\\
		&\Leftrightarrow  e^{-\frac{\la^\star}{\rho_1}}\dotp{\al}{e^{-\frac{\Bf}{\rho_1}}} = e^{+\frac{\la^\star}{\rho_2}}\dotp{\be}{e^{-\frac{\Bg}{\rho_2}}},\\
		&\Leftrightarrow -\frac{\la^\star}{\rho_1} + \log\dotp{\al}{e^{-\frac{\Bf}{\rho_1}}} = 
		\frac{\la^\star}{\rho_2} + \log\dotp{\be}{e^{-\frac{\Bg}{\rho_2}}},\\
		&\Leftrightarrow \la^\star\big(\frac{1}{\rho_1} + \frac{1}{\rho_2}\big) = 
		\log\Bigg[\frac{\dotp{\al}{e^{-\frac{\Bf}{\rho_1}}}}{\dotp{\al}{e^{-\frac{\Bg}{\rho_2}}}}\Bigg],\\
		&\Leftrightarrow \la^\star(\Bf,\Bg) = \frac{\rho_1\rho_2}{\rho_1 + \rho_2}
		\log\Bigg[\frac{\dotp{\al}{e^{-\frac{\Bf}{\rho_1}}}}{\dotp{\al}{e^{-\frac{\Bg}{\rho_2}}}}\Bigg].
	\end{align*}
	Hence the result of Proposition 2.
\end{proof}

\subsection{Proof of Proposition 3}

\begin{proof}
	Recall the expression of $\Ff_\epsilon$ which reads in the $\KL$ setting
	\begin{align*}
		\Ff_\epsilon(\f,\g) &= \dotp{\al}{-\rho_1(e^{-\f / \rho_1} - 1)} + \dotp{\be}{-\rho_2(e^{-\g / \rho_2} - 1)}\\
		&= \rho_1 m(\al) + \rho_2 m(\be) - \rho_1\dotp{\al}{e^{-\f / \rho_1}} - \rho_2\dotp{\be}{e^{-\g / \rho_2}}.
	\end{align*}
	We have that $\Hh_\epsilon(\Bf,\Bg) = \Ff_\epsilon(\Bf + \la^*(\Bf,\Bg), \Bg - \la^*(\Bf,\Bg))$.
	Applying Proposition 2, we have that
	\begin{align*}
		\dotp{\al}{e^{-\frac{\Bf + \la^\star(\Bf,\Bg)}{\rho_1}}} &= \dotp{\al}{e^{-\f / \rho_1}}\cdot\exp\Bigg(-\frac{\rho_2}{\rho_1 + \rho_2}
			\log\Bigg[\frac{\dotp{\al}{e^{-\frac{\Bf}{\rho_1}}}}{\dotp{\al}{e^{-\frac{\Bg}{\rho_2}}}}\Bigg]\Bigg)\\
			&= \dotp{\al}{e^{-\Bf / \rho_1}}\cdot\dotp{\al}{e^{-\Bf / \rho_1}}^{-\frac{\rho_2}{\rho_1 + \rho_2}}\cdot\dotp{\be}{e^{-\Bg / \rho_2}}^{\frac{\rho_2}{\rho_1 + \rho_2}}\\
			&=\dotp{\al}{e^{-\Bf / \rho_1}}^{\frac{\rho_1}{\rho_1 + \rho_2}}\cdot\dotp{\be}{e^{-\Bg / \rho_2}}^{\frac{\rho_2}{\rho_1 + \rho_2}}
	\end{align*}
	A similar calculation for $\dotp{\be}{e^{-(\Bg - \la^\star) / \rho_2}}$ yields
	\begin{align*}
		\dotp{\be}{e^{-\frac{\Bg - \la^\star(\Bf,\Bg)}{\rho_2}}} = \dotp{\al}{e^{-\frac{\Bf + \la^\star(\Bf,\Bg)}{\rho_1}}} =\dotp{\al}{e^{-\Bf / \rho_1}}^{\frac{\rho_1}{\rho_1 + \rho_2}}\cdot\dotp{\be}{e^{-\Bg / \rho_2}}^{\frac{\rho_2}{\rho_1 + \rho_2}}.
	\end{align*}
	Applying the above result in the definition of $\Hh_\epsilon$ yields the result of Proposition 3.
\end{proof}
\section{Appendix of Section 3 - Translation Invariant Sinkhorn algorithm}
\label{sec-supp-sinkhorn}

We focus in this appendix on detailing the properties of $\Hh$-Sinkhorn algorithm.
We recall belowe some notations:
\begin{align*}
	&\Psi_1: \Bf\mapsto \argmax_{\Bg} \Hh_\epsilon(\Bf,\Bg),\\
	&\Psi_2: \Bg\mapsto \argmax_{\Bf} \Hh_\epsilon(\Bf,\Bg),\\
	&\Phi: (\Bf,\Bg)\mapsto (\Bf + \la^\star(\Bf,\Bg), \Bg - \la^\star(\Bf,\Bg)),\\
	&\Upsilon_1: \Bf\mapsto (\Bf, \Psi_1(\Bf)),\\
	&\Upsilon_2: \Bg\mapsto (\Psi_2(\Bg), \Bg).
\end{align*}

In this section we focus on the properties on the map $\Psi_1$ (and $\Upsilon_1$) which represents the $\Hh$-Sinkhorn update of $\Bf$.
By analogy, those results hold for the map $\Psi_2$ which updates $\Bg$.

This section involves the use of several norms, namely the sup-norm $\norm{\cdot}_\infty$ and the Hilbert pseudo-norm $\norm{\f}_\star = \inf_{\la\in\RR} \norm{\f + \la}_\infty$ which is involved in the convergence study of the Balanced Sinkhorn algorithm~\cite{knight2008sinkhorn}.
The pseudo-norm $\norm{\cdot}_\infty$ is zero iff the functions are equal up to a constant
We also define a variant of the Hilbert norm defined for two functions $(\Bf,\Bg)$ which is definite up to the dual invariance $(\Bf + \la,\Bg-\la)$ of $\Hh$.
It reads $\norm{(\Bf,\Bg)}_{\star\star} \triangleq \min_{\la\in\RR} \norm{\Bf + \la}_\infty + \norm{\Bg-\la}_\infty$.

\subsection{Generic properties of $\Hh$-Sinkhorn updates}

\begin{prop}\label{prop-aprox-eq-psi}
	The map $\Psi_1(\Bf)$ satisfies the implicit equation
	\begin{align*}
	\Psi_1(\Bf) = -\aprox_{\phi^*_1}\Big(\; -\Smin{\al}{\epsilon}\big(\; \C - \Bf - \la^\star(\Bf, \Psi_1(\Bf)) \;\big) \;\Big) + \la^\star(\Bf, \Psi_1(\Bf)).
	\end{align*}
\end{prop}
\begin{proof}
	Recall the optimality Equation~\eqref{eq-optim-h-sink} transposed to optimality w.r.t. $\Bg$
	\begin{align}
	e^{\Bg/\epsilon}\dotp{\al}{e^{(\Bf-\C)/\epsilon}}=\nabla\phi_1^*(-\Bg+\la^\star(\Bf,\Bg)).
	\end{align}
	Perform a change of variable $\hg = \Bg - \la^\star(\Bf,\Bg)$, such that the above equation reads
	\begin{align}
	e^{\hg/\epsilon}\dotp{\be}{e^{(\Bf+\la^\star(\Bf,\Bg)-\C)/\epsilon}}=\nabla\phi_1^*(-\hg).
	\end{align}
	One can recognize the optimality condition of $\Ff$-Sinkhorn, thus one has
	\begin{align*}
		\hg = \Bg - \la^\star(\Bf,\Bg) = -\aprox_{\phi^*_1}\big(-\Smin{\al}{\epsilon}\big(\C - \Bf - \la^\star(\Bf,\Bg)\big)\big)
	\end{align*}
	Writing $\Bg = \Psi_1(\Bf)$ and adding on both sides $\la^\star(\Bf,\Bg)$ yields the result.
\end{proof}

\begin{prop}
	Assume $(\phi^*_1, \phi^*_2)$ are strictly convex. One has for any $\tau\in\RR$, $\la^\star(\Bf+\tau, \Bg)= \la^\star(\Bf,\Bg+\tau) - \tau$.
\end{prop}
\begin{proof}
	The strict convexity yields the uniqueness of $\la^\star$. Writing $\tilde{\la}=\la + \tau$, one has
	\begin{align*}
	\arg\max_{\la}\Ff_\epsilon(\Bf+\tau+\la, \Bg - \la)= \arg\max_{\tilde{\la}}\Ff_\epsilon(\Bf+\tilde{\la}, \Bg - \tilde{\la} + \tau) - \tau,
	\end{align*}
	Hence the desired relation.
\end{proof}

\begin{prop}\label{prop-aditivity-psi}
	One has $\Psi_1(\Bf + \tau) = \Psi_1(\Bf) - \tau$.
\end{prop}
\begin{proof}
	It is a combination of the previous two propositions, and the fact that the $\Hh$-update is uniquely defined.
\end{proof}

\subsection{Proof of Proposition 4 - Derivation of $\Hh$-Sinkhorn updates in the KL setting}
\label{app-proof-formula-h-sink}
\begin{proof}
We derive the $\Hh_\epsilon$-Sinkhorn optimality condition for $\Bf$ given $\Bg_t$, the equation for $\Bg$ are obtained by swapping the roles of $(\al,\Bf, \rho_1)$ and $(\be,\Bg, \rho_2)$.
Recall that thanks to Proposition 1, the optimality condition reads
\begin{align*}
e^{\Bf/\epsilon}\dotp{\be}{e^{(\Bg_t-\C)/\epsilon}}=\nabla\phi_1^*(-\Bf-\la^\star(\Bf,\Bg_t)) = e^{-(\Bf+\la^\star(\Bf,\Bg_t)) / \rho_1}.
\end{align*}
In the $\KL$ setting, thanks to Proposition 2, by taking the log one has
\begin{align*}
	&\frac{\Bf}{\epsilon} + \log\dotp{\be}{e^{(\Bg_t-\C)/\epsilon}} = -\frac{\Bf}{\rho_1} - \frac{\rho_2}{\rho_1 + \rho_2} \log\Big[\frac{\dotp{\al}{e^{-\Bf / \rho_1}}}{ \dotp{\be}{e^{-\Bg / \rho_2}}}\Big],\\
	&\Leftrightarrow \frac{\epsilon + \rho_1}{\epsilon\rho_1}\Bf + \frac{\rho_2}{\rho_1 + \rho_2} \log\dotp{\al}{e^{-\Bf / \rho_1}} = -\log\dotp{\be}{e^{(\Bg_t-\C)/\epsilon}} + \frac{\rho_2}{\rho_1 + \rho_2} \log \dotp{\be}{e^{-\Bg_t / \rho_2}}.
\end{align*}
We recall the definition of the Softmin $\Smin{\al}{\epsilon}(\f) \triangleq -\epsilon\log\dotp{\al}{e^{-\f / \epsilon}}$.
 An important property used here is that for any $\tau\in\RR$, one has $\Smin{\al}{\epsilon}(\f + \tau) = \Smin{\al}{\epsilon}(\f) + \tau$.
 The calculation then reads
\begin{align*}
	&\frac{\epsilon + \rho_1}{\epsilon\rho_1}\Bf - \frac{\rho_2}{\rho_1(\rho_1 + \rho_2)}\Smin{\al}{\rho_1}(\Bf) = \frac{1}{\epsilon}\Smin{\be}{\epsilon}(\C-\Bg_t) - \frac{1}{\rho_1 + \rho_2}\Smin{\be}{\rho_2}(\Bg_t),\\
	&\Leftrightarrow \Bf - \frac{\epsilon}{\epsilon+\rho_1}\cdot\frac{\rho_2}{\rho_1 + \rho_2}\Smin{\al}{\rho_1}(\Bf) = \frac{\rho_1}{\rho_1 + \epsilon}\Smin{\be}{\epsilon}(\C-\Bg_t) - \frac{\epsilon}{\epsilon+\rho_1}\cdot\frac{\rho_1}{\rho_1 + \rho_2}\Smin{\be}{\rho_2}(\Bg_t).
\end{align*}
We now define the function $\hf_{t+1}$ as
\begin{align*}
	\hf_{t+1} \triangleq \frac{\rho_1}{\rho_1 + \epsilon}\Smin{\be}{\epsilon}(\C-\Bg_t) - \frac{\epsilon}{\epsilon+\rho_1}\cdot\frac{\rho_1}{\rho_1 + \rho_2}\Smin{\be}{\rho_2}(\Bg_t),
\end{align*}
Such that the optimality equation now reads
\begin{align*}
\Bf - \frac{\epsilon}{\epsilon+\rho_1}\cdot\frac{\rho_2}{\rho_1 + \rho_2}\Smin{\al}{\rho_1}(\Bf) = \hf_{t+1}.
\end{align*}
Define $k \triangleq \tfrac{\epsilon}{\epsilon+\rho_1}\cdot\tfrac{\rho_2}{\rho_1 + \rho_2}$.
Note that $\frac{\epsilon}{\epsilon+\rho_1}\cdot\frac{\rho_2}{\rho_1 + \rho_2}\Smin{\al}{\rho_1}(\Bf)$ is in $\RR$, thus there exists some $\tau\in\RR$ such that $\Bf = \hf_{t+1} + \tau$.
Using such property in the above equation yields
\begin{align*}
	&\hf_{t+1} + \tau - k\Smin{\al}{\rho_1}(\hf_{t+1} + \tau) = \hf_{t+1},\\
	&\Leftrightarrow \tau - k\big(\Smin{\al}{\rho_1}(\hf_{t+1}) + \tau\big) = 0,\\
	&\Leftrightarrow \tau(1 - k) = k\Smin{\al}{\rho_1}(\hf_{t+1}),\\
	& \Leftrightarrow \tau = \frac{k}{(1 - k)}\Smin{\al}{\rho_1}(\hf_{t+1}).
\end{align*}
We can conclude and say that
\begin{align}\label{eq-general-iter-h-sink}
	\Bf = \hf_{t+1} + \tau = \hf_{t+1} + \frac{k}{(1 - k)}\Smin{\al}{\rho_1}(\hf_{t+1})
\end{align}
Note that we retrieve the map of Proposition 4 in the case $\rho_1=\rho_2=\rho$.
Indeed one hase the simpification
\begin{align*}
	\frac{k}{1-k} &= \frac{\epsilon\rho}{2\rho(\epsilon + \rho) - \epsilon\rho}\\
	& = \frac{\epsilon}{2(\epsilon + \rho) - \epsilon}\\
	&= \frac{\epsilon}{\epsilon + 2\rho}.
\end{align*}
Thus when $\rho_1=\rho_2=\rho$ the full iteration from $\Bg_t$ to $\Bf_{t+1}$ reads
\begin{align*}
	\hf_{t+1} &= \frac{\rho}{\rho + \epsilon}\Smin{\be}{\epsilon}(\C-\Bg_t) - \frac{1}{2}\frac{\epsilon}{\epsilon+\rho}\Smin{\be}{\rho}(\Bg_t),\\
	\Bf_{t+1} &= \hf_{t+1} + \frac{\epsilon}{\epsilon + 2\rho}\Smin{\al}{\rho_1}(\hf_{t+1}).
\end{align*}

\end{proof}

We reformulate the full update $\Psi_1$ with a single formula instead of the above two formulas.
While the above formulas formalize the most convenient way to implement it (because we only store one vector of length N at any time), the following result will be more convenient to derive a convergence analysis.

\begin{prop}\label{prop-full-formula-psi}
	Assume $\rho_1 = \rho_2 = \rho$. One has
	\begin{align*}
	\Psi_1(\Bf) = \tfrac{\rho}{\rho+\epsilon}\Smin{\al}{\epsilon}(\C-\Bf) 
	+ \tfrac{\epsilon}{\epsilon + 2\rho} \Big( \Smin{\be}{\rho}(\tfrac{\rho}{\rho+\epsilon}\Smin{\al}{\epsilon}(\C-\Bf)) - \Smin{\al}{\rho}(\Bf) \Big). 
	\end{align*}
	Furthermore one has $\Psi_1(\Bf + \la) = \Psi_1(\Bf) - \la$ for any $\la\in\RR$.
\end{prop}

\subsection{Properties on $\Hh$-Sinkhorn updates in the KL setting}

We now focus on the setting of $\KL$ penalties to derive sharper results on the convergence of $\Hh$-Sinkhorn.

\begin{prop}\label{prop-psi-nonexp}
	In the $\KL$ setting with parameters $(\rho_1,\rho_2)$, the operator $\Psi_1$ is non-expansive for the sup-norm $\norm{\cdot}_\infty$, i.e for any $(\Bf,\Bg)$, one has
	\begin{align*}
		\norm{\Psi_1(\Bf) - \Psi_1(\Bg)}_\infty \leq \norm{\Bf - \Bg}_\infty
	\end{align*}
\end{prop}
\begin{proof}
	We use the formulas given in the proof of Proposition~\ref{prop-conv-h-sink}, in Appendix~\ref{app-proof-formula-h-sink}, and reuse the same notations. Recall from~\cite{sejourne2019sinkhorn} that one has for any measure $\al$, cost $\C$, parameter $\epsilon,\rho >0$
	\begin{align*}
		\norm{\Smin{\al}{\epsilon}(\C-\Bf) - \Smin{\al}{\epsilon}(\C-\Bg)}_\infty &\leq \norm{\Bf - \Bg}_\infty\\
		\norm{\Smin{\al}{\rho}(\Bf) - \Smin{\al}{\rho}(\Bg)}_\infty &\leq \norm{\Bf - \Bg}_\infty.
	\end{align*}
	By chaining those inequalities in the quantity $\norm{\Psi_1(\Bf) - \Psi_1(\Bg)}_\infty$ (By using Equation~\eqref{eq-general-iter-h-sink}), and because $\norm{\kappa\Bf}_\infty=\kappa\norm{\Bf}_\infty$ for any $\kappa\in\RR_+$ , one gets
	\begin{align*}
		\norm{\Psi_1(\Bf) - \Psi_1(\Bg)}_\infty \leq \Xi \norm{\Bf - \Bg}_\infty,
	\end{align*}
	where
	%
	%
	\begin{align*}
	\Xi &= \Bigg[ \frac{\rho_1}{\epsilon + \rho_1} + \frac{\epsilon}{\epsilon + \rho_1}\frac{\rho_1}{\rho_1 + \rho_2} \Bigg]
	 	 + \frac{k}{1 - k} \Bigg[ \frac{\rho_1}{\epsilon + \rho_1} + \frac{\epsilon}{\epsilon + \rho_1}\frac{\rho_1}{\rho_1 + \rho_2} \Bigg]\\
	 	 &=\Bigg[ \frac{\rho_1}{\epsilon + \rho_1} + \frac{\epsilon}{\epsilon + \rho_1}\frac{\rho_1}{\rho_1 + \rho_2} \Bigg]\Big(1 + \frac{k}{1 - k}\Big)\\
	 	 &=\Bigg[ \frac{\rho_1}{\epsilon + \rho_1} + \frac{\epsilon}{\epsilon + \rho_1}\frac{\rho_1}{\rho_1 + \rho_2} \Bigg]\Big(\frac{1}{1 - k}\Big).
	\end{align*}
	A calculation yields
	\begin{align*}
		\frac{\rho_1}{\epsilon + \rho_1} + \frac{\epsilon}{\epsilon + \rho_1}\frac{\rho_1}{\rho_1 + \rho_2} &= 
		\frac{\rho_1(\epsilon + \rho_1 + \rho_2)}{(\epsilon + \rho_1)(\rho_1 + \rho_2)},\\
		\frac{1}{1 - k} &= \frac{(\epsilon + \rho_1)(\rho_1 + \rho_2)}{\rho_1(\epsilon + \rho_1 + \rho_2)}.
	\end{align*}
	Thus $\Xi=1$, hence the nonexpansive property.
\end{proof}

We provide below additional details on the non-expansiveness of the map $\Phi$.

\begin{prop}\label{prop-Phi-nonexpansive}
	The map $\Phi:(\Bf,\Bg)\mapsto(\Bf + \la^\star(\Bf,\Bg), \Bg - \la^\star(\Bf,\Bg))$, is $1$-Lipschitz from the norm $\norm{(\f,\g)}_\infty$ to $\norm{(\Bf,\Bg)}_{\star\star}$, i.e. one has
	\begin{align*}
	\norm{\Bf_1 + \la^\star(\Bf_1,\Bg_1) - \Bf_2 - \la^\star(\Bf_2,\Bg_2)}_\infty
	+ \norm{\Bg_1 - \la^\star(\Bf_1,\Bg_1) - \Bg_2 + \la^\star(\Bf_2,\Bg_2)}_\infty
	\leq \norm{(\Bf_1,\Bg_1) - (\Bf_2,\Bg_2)}_{\star\star},
	\end{align*}
	where $\norm{(\Bf,\Bg)}_{\star\star} \triangleq \min_{\la\in\RR} \norm{\Bf + \la}_\infty + \norm{\Bg-\la}_\infty.$
\end{prop}
\begin{proof}
%
%

To prove such statement, first note that we can rewrite $\la^\star$ as
\begin{align*}
	\la^\star(\Bf, \Bg) = \tau_1\Smin{\be}{\rho_2}(\Bg) - \tau_2\Smin{\al}{\rho_1}(\Bf),
\end{align*}
where $\tau_1=\tfrac{\rho_1}{\rho_1 + \rho_2}$ and $\tau_2=\tfrac{\rho_2}{\rho_1 + \rho_2}$.

We consider now that $(\Bf,\Bg)$ are discrete and concatenated to form a vector of size $N+M$.
Note that the gradient of $\Smin{\al}{\rho_1}(\Bf)$ reads 
\begin{align*}
	\nabla\Smin{\al}{\rho_1}(\Bf)_i = \frac{e^{-\Bf_i / \rho_1}\al_i}{\sum_k e^{-\Bf_k / \rho_1}\al_k}.
\end{align*}

Using this formula we can compute the Jacobian of $\Phi$, noted $J\Phi(\Bf,\Bg)$. It reads
\begin{align*}
	J\Phi(\Bf,\Bg) = 
	\begin{pmatrix}
	I_N - \tau_2\nabla\Smin{\al}{\rho_1}(\Bf)\ones_N^\top & \tau_1\nabla\Smin{\be}{\rho_2}(\Bg)\ones_M^\top\\
	\tau_2\nabla\Smin{\al}{\rho_1}(\Bf)\ones_N^\top & I_M - \tau_1\nabla\Smin{\be}{\rho_2}(\Bg)\ones_M^\top
	\end{pmatrix}.
\end{align*}
A key property of the Jacobian is that for any $\la\in\RR$, one has $J\Phi(\Bf,\Bg)^\top(\la\ones_N,-\la\ones_M)=0$.
%
%
We now derive the Lipschitz bound. We define $(\Bf_t,\Bg_t) = (\Bf_1 + t(\Bf_2 - \Bf_1), \Bg_1 + t(\Bg_2 - \Bg_1))$.
The computation reads
\begin{align*}
	\norm{\Phi(\Bf_1,\Bg_1) - \Phi(\Bf_2, \Bg_2)}_\infty &= \norm{ \int_0^1 \frac{\d\Phi(\Bf_t,\Bg_t)}{\d t}\d t }_\infty\\
	&=\norm{J\Phi(\Bf_t,\Bg_t)^\top\big((\Bf_2,\Bg_2) - (\Bf_1,\Bg_1)\big)}_\infty\\
	&=\norm{J\Phi(\Bf_t,\Bg_t)^\top\big((\Bf_2,\Bg_2) - (\Bf_1,\Bg_1) + (\la\ones_N,-\la\ones_M) \big) }_\infty\\
	&\leq\norm{J\Phi(\Bf_t,\Bg_t)}_\infty 
	\Big(\min_{\la\in\RR} \norm{\Bf_2 - \Bf_1 + \la}_\infty + \norm{\Bg_2 - \Bg_1-\la}_\infty\Big)\\
	&\leq\norm{J\Phi(\Bf_t,\Bg_t)}_\infty \norm{(\Bf_2,\Bg_2) - (\Bf_1,\Bg_1)}_{\star\star}.
\end{align*}
Since $\norm{J\Phi(\Bf_t,\Bg_t)}_\infty\leq 1$, we get the desired Lipschitz property. 
%
\end{proof}

We define $\Upsilon_1(\Bf) \triangleq (\Bf, \Psi_1(\Bf))$ where $\Psi_1$ is detailed in Proposition~\ref{prop-conv-h-sink}.
Similarly, one can define $\Upsilon_2(\Bg) \triangleq (\Psi_2(\Bg), \Bg)$.
We present properties for $\Upsilon_1$, but they analogously hold for $\Upsilon_2$.

\begin{prop}\label{prop-conj-redisual}
	Consider any functions $(\Bf,\Bg)$. One has
	\begin{align*}
	\norm{\Upsilon_1(\Bf) - \Upsilon_1(\Bg)}_{\star\star} \leq 2\norm{\Bf - \Bg}_\star
	\end{align*}
\end{prop}
\begin{proof}
	By definition of $\norm{\cdot}_{\star\star}$, one has
	\begin{align*}
	\norm{\Upsilon_1(\Bf) - \Upsilon_1(\Bg)}_{\star\star}
			&=\norm{(\Bf, \Psi_1(\Bf)) - (\Bg, \Psi_1(\Bg))}_{\star\star}\\
			&= \inf_{\la\in\RR} \norm{\Bf - \Bg + \la}_\infty + \norm{\Psi_1(\Bf) - \Psi_1(\Bg) - \la}_\infty\\
			&= \inf_{\la\in\RR} \norm{(\Bf + \la) - \Bg}_\infty + \norm{\Psi_1(\Bf + \la) - \Psi_1(\Bg)}_\infty\\
			&\leq 2 \inf_{\la\in\RR} \norm{(\Bf + \la) - \Bg}_\infty\\
			&= 2\norm{\Bf - \Bg}_\star,
	\end{align*}
	where we use the relation $\Psi_1(\Bf + \la) = \Psi_1(\Bf) - \la$ from Proposition~\ref{prop-aditivity-psi}, and where the inequality is given by Proposition~\ref{prop-psi-nonexp}.
	Hence we get the desired bound, which ends the proof.
\end{proof}

Before providing a convergence result, we detail the contraction properties of the map $\Psi_1$ w.r.t the Hilbert norm $\norm{\cdot}_\star$.

\begin{prop}\label{prop-psi-contractive-hilbert}
	Consider two functions $(\Bf,\Bg)$, one has
	\begin{align*}
		\norm{\Psi_1(\Bf) - \Psi_1(\Bg)}_\star \leq \frac{\rho}{\epsilon + \rho} \kappa_\epsilon(\al)\norm{\Bf - \Bg}_\star,
	\end{align*}
	where $\kappa_\epsilon(\al)<1$ is the contraction constant of the Softmin~\cite{knight2008sinkhorn} which reads
	\begin{align*}
		\norm{\Smin{\al}{\epsilon}(\Bf) - \Smin{\al}{\epsilon}(\Bg)}_\star \leq \kappa_\epsilon(\al)\norm{\Bf - \Bg}_\star.
	\end{align*}
\end{prop}
\begin{proof}
	Thanks to Proposition~\ref{prop-full-formula-psi}, we have
	\begin{align*}
	\Psi_1(\Bf) &= \tfrac{\rho}{\rho+\epsilon}\Smin{\al}{\epsilon}(\C-\Bf) 
	+ \tfrac{\epsilon}{\epsilon + 2\rho} \Big( \Smin{\be}{\rho}(\tfrac{\rho}{\rho+\epsilon}\Smin{\al}{\epsilon}(\C-\Bf)) - \Smin{\al}{\rho}(\Bf) \Big)\\
	&= \tfrac{\rho}{\rho+\epsilon}\Smin{\al}{\epsilon}(\C-\Bf) + T(\Bf),
	\end{align*}
	where $T(\Bf)\in\RR$ is a constant translation (Note that the only term outputing a function is the one involving $\C(x,y)$).
	Thus because the Hilbert norm is invariant to translations, one has
	\begin{align*}
		\norm{\Psi_1(\Bf) - \Psi_1(\Bg)}_\star 
		= \norm{\tfrac{\rho}{\rho+\epsilon}\Smin{\al}{\epsilon}(\C-\Bf) - \tfrac{\rho}{\rho+\epsilon}\Smin{\al}{\epsilon}(\C-\Bg)}_\star.
	\end{align*}
	Thanks to the results of~\cite{chizat2016scaling, knight2008sinkhorn}, one has
	\begin{align*}
		\norm{\tfrac{\rho}{\rho+\epsilon}\Bf - \tfrac{\rho}{\rho+\epsilon}\Bg}_\star &= \tfrac{\rho}{\rho+\epsilon}\norm{\Bf - \Bg}_\star,\\
		\norm{\Smin{\al}{\epsilon}(\Bf) - \Smin{\al}{\epsilon}(\Bg)}_\star &\leq \kappa_\epsilon(\al)\norm{\Bf - \Bg}_\star,
	\end{align*}
	Which ends the proof of the contraction property of $\Psi_1$.
\end{proof}

We prove now a convergence result of the $\Hh$-Sinkhorn algorithm, before providing a quantitative rate.

\begin{thm}
	The map $\Psi_2\circ\Psi_1$ converges to a fixed point $\Bf^\star$ such that $\Bf^\star=\Psi_2\circ\Psi_1(\Bf^\star)$, where $\Bf$ is defined up to a translation.
	The function $\Bg^\star=\Psi_1(\Bf^\star)$ also satisfies $\Bg^\star=\Psi_1\circ\Psi_2(\Bg^\star)$.
	Furthermore the functions $(\f^\star,\g^\star)=\Phi(\Bf^\star,\Bg^\star)$ are fixed points of the $\Ff$-Sinkhorn updates, and are thus optimizers of the functional $\Ff$.
\end{thm}
\begin{proof}
	Thanks to Proposition~\ref{prop-psi-contractive-hilbert}, we know that the map $\Psi_2\circ\Psi_1$ is contractive for the Hilbert norm, thus there is uniqueness of the fixed point.
	Because the map $\Psi_2\circ\Psi_1(\Bf + \la) = \Psi_2\circ\Psi_1(\Bf) + \la$, one can assume without loss of generality that all iterates $\Bf_t$ satisfy $\Bf_t(x_0)=0$ for some $x_0$ in the support of $\al$.
	Thus under similar assumptions as in~\cite{sejourne2019sinkhorn}, we get that the iterates lie in a compact set, which yields existence of a fixed point $\Bf^\star$ satisfying $\Psi_2\circ\Psi_1(\Bf^\star)=\Bf^\star$.
	Defining $\Bg^\star=\Psi_1(\Bf^\star)$ and composing the previous relation by $\Psi_1$, we get $\Bg^\star=\Psi_1\circ\Psi_2(\Bg^\star)$.

	Recall from Proposition~\ref{prop-aprox-eq-psi} that $\Bf^\star$ satisfies the relation
	\begin{align*}
	\Bg^\star = \Psi_1(\Bf^\star) = \frac{\rho}{\epsilon+\rho}\Big(\; \Smin{\al}{\epsilon}\big(\; \C - \Bf^\star - \la^\star(\Bf^\star, \Bg^\star) \;\big) \;\Big) + \la^\star(\Bf^\star, \Bg^\star).
	\end{align*}
	Thus defining $(\f^\star,\g^\star) = \Phi(\Bf^\star,\Bg^\star) = (\Bf^\star +  \la^\star(\Bf^\star, \Bg^\star),\Bg^\star -  \la^\star(\Bf^\star, \Bg^\star))$, the above equation can be rephrased as
	\begin{align*}
	\g^\star = \frac{\rho}{\epsilon+\rho}\Big(\; \Smin{\al}{\epsilon}\big(\; \C - \f^\star \;\big) \;\Big),
	\end{align*}
	which is exactly the fixed point equation of $\Ff$-Sinkhorn, thus $(\f^\star,\g^\star)$ are optimal dual potentials for $\Ff$.
\end{proof}

Based on the above results, we can prove the following convergence rate

\begin{thm}
	Write $\Bf^\star$ the fixed point of the map $\Psi_2\circ\Psi_1$. Take $\Bf_{t}$ obtained by $t$ iterations of the map $\Psi_2\circ\Psi_1$, starting from the function $\f_0$. One has
	\begin{align*}
	\norm{\Phi(\Upsilon_1(\Bf)(\Bf_t)) - \Phi(\Upsilon_1(\Bf)(\Bf^\star))}_\infty 
	\leq  \norm{\Upsilon_1(\Bf)(\Bf_t) - \Upsilon_1(\Bf)(\Bf^\star)}_{\star\star}
	\leq 2\norm{\Bf_t - \Bf^\star}_\star
	\leq 2\bar{\kappa}^{t}\norm{\Bf_0 - \Bf^\star}_\star,
	\end{align*}
	where $\bar{\kappa}\triangleq (1 + \tfrac{\epsilon}{\rho_1})^{-1}\kappa_\epsilon(\al)(1 + \tfrac{\epsilon}{\rho_2})^{-1}\kappa_\epsilon(\be)$.
\end{thm}
\begin{proof}
	\begin{itemize}
		\item The first inequality is proved in Proposition~\ref{prop-Phi-nonexpansive}.
		\item The second inequality is given by Proposition~\ref{prop-conj-redisual}
		\item The last inequality is obtained by applying Proposition~\ref{prop-psi-contractive-hilbert} to both $\Psi_1$ and $\Psi_2$ consecutively to get the contraction rate $\bar{\kappa}$ for any $(\rho_1,\rho_2)$, i.e. we have
		\begin{align*}
			\norm{\Bf_t - \Bf^\star}_\star &=\norm{\Psi_2\circ\Psi_1(\Bf_{t-1}) - \Psi_2\circ\Psi_1(\Bf^\star)}_\star\\
			&\leq \bar{\kappa}\norm{\Bf_{t-1} - \Bf^\star}_\star.
		\end{align*}
		By iterating this bound by induction over all iterations, we get last bound of the statement, which ends the proof.
	\end{itemize}
\end{proof}

\subsection{Interesting norms and closed formulas}

We study now properties of the norm $\norm{(\Bf,\Bg)}_{\star\star} \triangleq \min_{\la\in\RR} \norm{\Bf + \la}_\infty + \norm{\Bg-\la}_\infty$.
It will share connections with the Hilbert norm $\norm{\Bf}_\star \triangleq \min_{\la\in\RR} \norm{\Bf + \la}_\infty$.
We firts prove the following Lemma on the Hilbert norm%
\begin{lem}\label{lem-hilbert-closed-form}
	One has $\norm{\Bf}_\star=\tfrac{1}{2}(\max\Bf - \min\Bf)$.
\end{lem}
\begin{proof}
	We will use the relations $\norm{\Bf}_\infty = \max(\max\Bf, -\min\Bf)$ and $\max(x,y)=\tfrac{1}{2}(x + y + |x-y|)$.%
	Applying those relations to the setting of the Hilbert norm yields for any $\la\in\RR$
	\begin{align*}
		\norm{\Bf + \la}_\infty &=  \max(\max\Bf + \la, -\min\Bf - \la)\\
		&= \tfrac{1}{2}(\max\Bf - \min\Bf + |\max\Bf + \min\Bf + 2\la|).
	\end{align*}
	Since the Hilbert norm is obtained by minimizing over $\la\in\RR$, we see that the minimum is attained at the unique value $\la^\star=\tfrac{1}{2}(\max\Bf + \min\Bf)$.
	Thus the absolute value cancels out and it yields $\norm{\Bf}_\star=\tfrac{1}{2}(\max\Bf - \min\Bf)$.
\end{proof}
We focus now on $\norm{(\Bf,\Bg)}_{\star\star}$.
\begin{lem}\label{lem-starstar-closed-form}
	One has $\norm{(\Bf,\Bg)}_{\star\star} = \norm{\Bf}_\star + \norm{\Bg}_\star + \tfrac{1}{2}|\max\Bf + \min\Bf + \max\Bg + \min\Bg| = \norm{\Bf\oplus\Bg}_\infty$.
\end{lem}
\begin{proof}
	The proof reuses the same relations used in the previous lemma. One has for any $\la\in\RR$
	\begin{align*}
		\norm{\Bf + \la}_\infty + \norm{\Bg-\la}_\infty &= \max(\max\Bf + \la,-\min\Bf - \la) + \max(\max\Bg-\la,-\min\Bg + \la)\\
		&= \tfrac{1}{2}(\max\Bf - \min\Bf + |\max\Bf + \min\Bf + 2 \la|) + \tfrac{1}{2}(\max\Bg - \min\Bg + |\max\Bg + \min\Bg - 2 \la|).
	\end{align*}
	Thanks to Lemma~\ref{lem-scalar-solve-starstar}, we have that the minimization in $\la$ attains the following value
	\begin{align*}
		\norm{(\Bf,\Bg)}_{\star\star} &= \tfrac{1}{2}(\max\Bf - \min\Bf) + \tfrac{1}{2}(\max\Bg - \min\Bg) + \tfrac{1}{2}|\max\Bf + \min\Bf + \max\Bg + \min\Bg|.
	\end{align*}
	Reusing Lemma~\ref{lem-hilbert-closed-form} allows to rewrite the first two terms as Hilbert norms.
	To get the last equality with $\norm{\Bf\oplus\Bg}_\infty$, note that $\max(\Bg\oplus\Bg) = \max\Bf + \max\Bg$, and that the above relation reads $max(x,y)$ with $x= \max\Bf + \max\Bg$ and $y=-\min\Bf-\min\Bg$.
\end{proof}
We end with the proof of the Lemma we used in the previous demonstration.
\begin{lem}\label{lem-scalar-solve-starstar}
	For any $(a,b)\in\RR$, one has $\min_{\la\in\RR} |a + \la| + |b - \la| = |a + b|$, which is attained for any $\la\in[\min(-a,b), \max(-a,b)]$.
\end{lem}
\begin{proof}
	Given that $a\leq b$, we study all cases depending on the sign of the absolute value.
	We define $V(\la)\triangleq |a + \la| + |b - \la|$.
	\paragraph{Case 1: $(a+\la\geq 0)$ and $(b-\la\geq 0)$.}
	It is equivalent to $\la\in[-a, b]$, which is non-empty when $\min(-a,b) = -a$ and $\max(-a,b) = b$.
	One has $V(\la) = a + \la + b - \la = a + b = |a + b|$ because in this case $-a\leq b$.
	\paragraph{Case 2: $(a+\la\leq 0)$ and $(b-\la\leq 0)$.}
	It is equivalent to $\la\in[b, -a]$, which is non-empty when $\min(-a,b) = b$ and $\max(-a,b) = -a$.
	One has $V(\la) = -a - \la - b + \la = -(a + b)= |a + b|$ because in this case $-a\geq b$.
	\paragraph{Case 3: $(a+\la\geq 0)$ and $(b-\la\leq 0)$.}
	It is equivalent to $\la\in[\max(-a, b), +\infty)$.
	One has $V(\la) = a + \la - b + \la = a - b + 2\la$, which is minimized for $\la=\max(-a, b)$.
	It yields $$V(\la)=a - b + 2\max(-a, b) = a - b + (b - a + |a + b|) = |a + b|.$$
	\paragraph{Case 4: $(a+\la\leq 0)$ and $(b-\la\geq 0)$.}
	It is equivalent to $\la\in](-\infty, \min(-a, b)]$.
	One has $V(\la) = -a - \la + b - \la = b - a - 2\la$, which is minimized at $\la=\min(-a,b)$.
	It yields $$V(\la)=b - a - 2\min(-a, b) = b - a - (b - a - |a + b|) = |a + b|.$$
	
	\paragraph{Conclusion.}
	All in all, we have from all cases altogether that the $\min_{\la\in\RR}V(\la) = |a + b|$, and that any $\la\in[\min(-a,b), \max(-a,b)]$ attains this minimum.
\end{proof}

\subsection{Experiments - Combining Sinkhorn with Anderson acceleration}

Anderson acceleration~\cite{anderson1965iterative} applies to any iterative map of the form $x_{t+1} = T(x_t)$ for $x_t\in\RR^d$.
At a given iterate $x_t$ consists in storing $K$ residuals of the form $u_k=T(x_{t+k}) - x_{t+k}$ for $k=0..K-1$ as a matrix $U \triangleq [u_0, \ldots, u_{K-1}]$, and to find the best interpolation of the residuals which satisfy the following criteria
\begin{align*}
	c^\star\in\arg\min_{\ones^\top c = 1} \norm{Uc}_2,
\end{align*}
where $c\in\RR^K$
Then one defines the next iterate as $x_{t+1} = \sum_{k=0}^{K-1} c_k x_{t+k}$.
Such procedure is known to converge faster to a fixed point than the standard iterations~\cite{scieur2016regularized}.
To ensure the convergence and the well-posedness of $c^\star$, it is common to regularize the problem as
\begin{align*}
	c^\star\in\arg\min_{\ones^\top c = 1} c^\top(U^\top U + r I)c,
\end{align*}
where $r\geq 0$ is the regularization parameter.
In this case we have the following closed form
\begin{align*}
	c^\star_r \triangleq \frac{(U^\top U + r I)^{-1}\ones}{\ones^\top(U^\top U + r I)^{-1}\ones}.
\end{align*}
The interest of Anderson acceleration is that the above extrapolation amounts to invert a small matrix of size $K\times K$.

We provide below an experiment on the estimation of the contraction rate similar to Figure 1.
We take $K=4$ and $r=10^{-7}$.
We test the acceleration on each version of Sinkhorn, and we observe that it yields a faster convergence for all three versions $(\Ff_\epsilon,\Gg_\epsilon,\Hh_\epsilon)$- Sinkhorn.

\begin{figure}[H]
	\centering
	\includegraphics[width=0.4\textwidth]{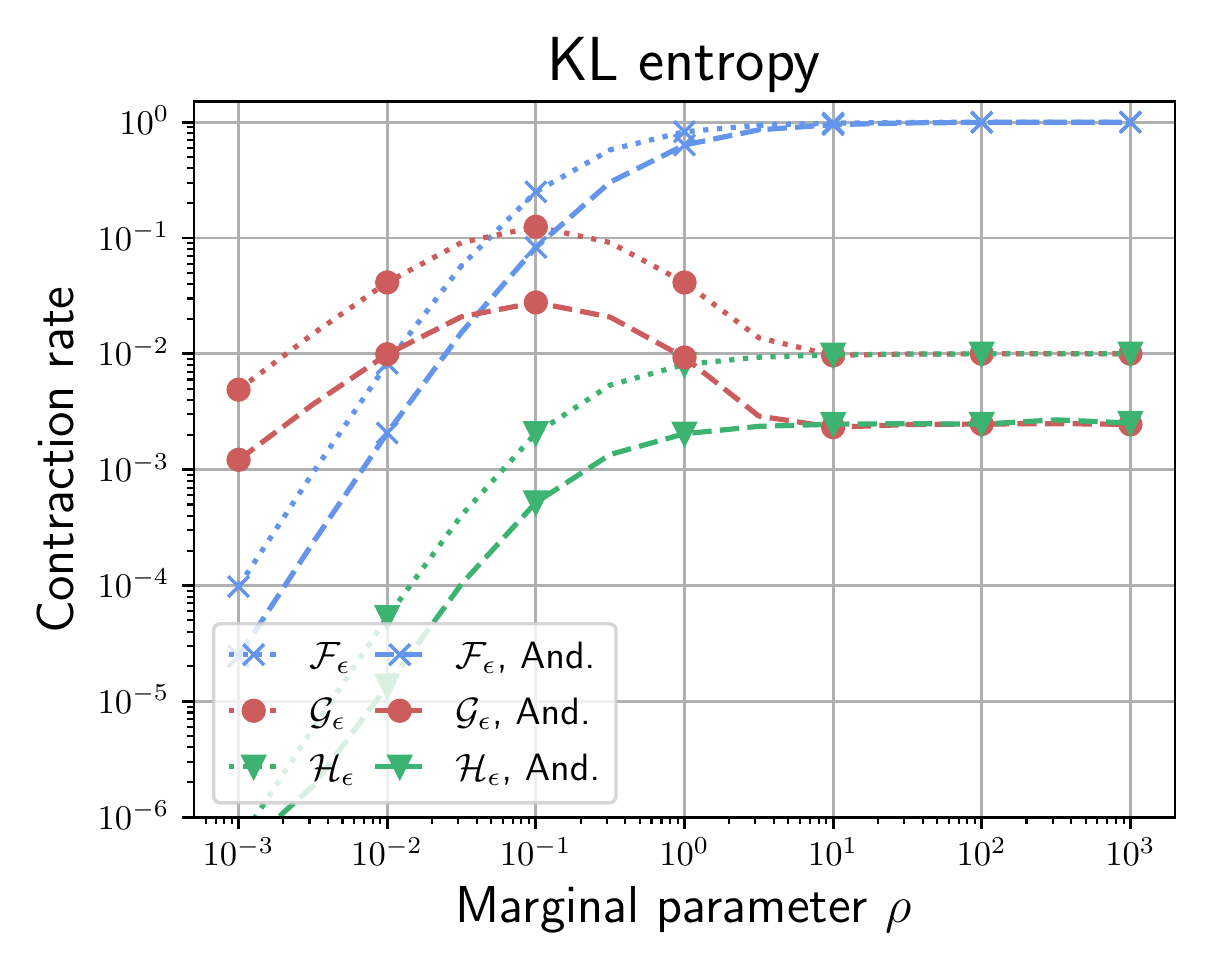}
	\caption{\textit{
			Estimation of the contraction for the Anderson acceleration applied to $(\Ff_\epsilon,\Gg_\epsilon,\Hh_\epsilon)$- Sinkhorn compared with the contraction rate of $(\Ff_\epsilon,\Gg_\epsilon,\Hh_\epsilon)$-Sinkhorn. Dashed lines represent the accelerated version and dotted lines the 'standard' algorithm. We take $K=4$ iterations for Anderson extrapolation, and we regularize with $r=10^{-7}$.}
	}
\end{figure}

\section{Appendix of Section 4 - Frank-Wolfe solver in 1-D}
\label{seq-supp-fw}

\subsection{Details on the Pairwise Frank-Wolfe algorithm}

Frank-Wolfe or conditional gradient methods~\cite{frank1956algorithm} aims at minimizing $\min_{x\in conv(\Aa)} \Ff(x)$, where $\Aa$ is called the set of atoms.
To do so one can minimize a linear minimization oracle (LMO) at the current iterate $x_t$, which reads $v_t\in\arg\min_{v\in\Aa} \dotp{\nabla\Ff(x_t)}{v}$.
The next iterate $x_{t+1}$ is updated as a convex combination of $(x_t,v_t)$ via line-search $\gamma_t\in\arg\min_{\gamma\in[0,1]} \Ff(x_t + \gamma d_t)$, where $d_t=v_t-x_t$ is the descent direction.
It is also possible to skip the line-search and set $\gamma_t=\tfrac{2}{2+t}$, which gives a $O(\tfrac{1}{t})$ approximation rate of the optimizer for gradient-Lipschitz functions.

We also consider in this paper the paiwise FW (PFW) variant~\cite{lacoste2015global}.
We store at each iterate the atom $v_t$ and its weight $w_t\geq 0$ in the convex combination as a dictionnary $\Vv_t$, i.e. at time $t$ one has $x_t = \sum_{k=1}^t w_k v_k$, and $\sum_{k=1}^t w_k = 1$.
What changes is the descent direction in the linesearch $d_t=v_t-s_{t^\star}$ where $s_{t^\star}\in\arg\max_{s\in\Vv_t}\dotp{\nabla\Ff(x_t)}{s}$.
The linesearch seeks $\gamma\in[0,w_{t^\star}]$ instead of $[0,1]$ to ensure that $x_{t+1}=x_t + \gamma d_t$ remains a convex combination.
One can interpret this variant as removing previous atoms which became irrelevant to replace them with more optimal ones.
There is an affine-invariant analysis of this variant in~\cite{lacoste2015global} which ensures linear convergence under milder assumptions than the conditions of FW with step $\gamma_t=\tfrac{2}{2+t}$.

\subsection{Additional figures on the comparison FW v.s. Sinkhorn}

We provide below a Figure which is in the same setting as Figure 4, except that the value is set to $\rho=10^{-1}$ and $\rho=10$.
The results are represented in the Figure below.
\begin{figure*}[h!]
	\centering
	\begin{tabular}{c@{}c@{}c}
		{\includegraphics[width=0.3\linewidth]{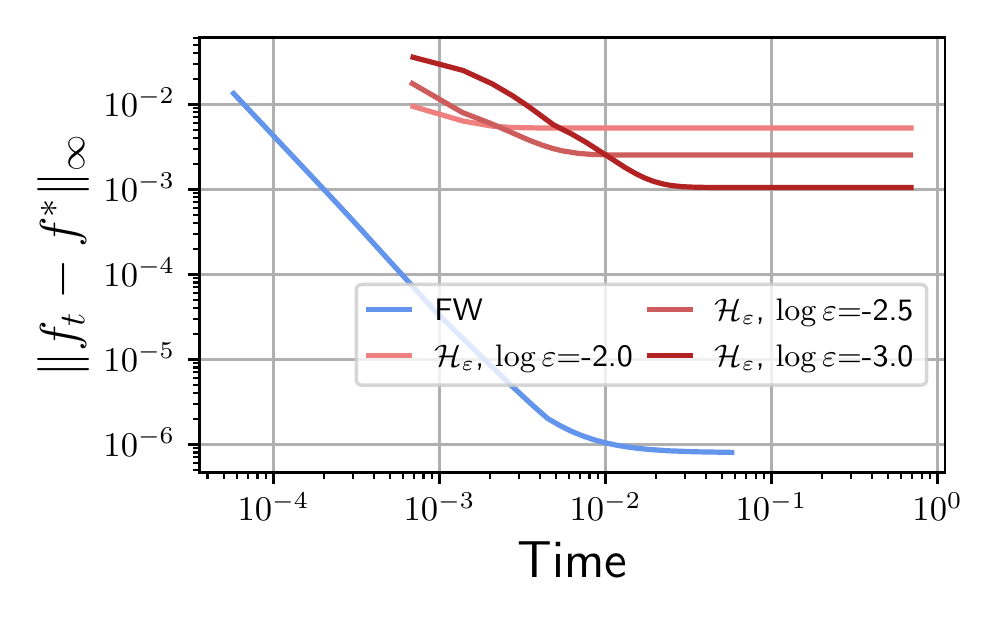}} &
		{\includegraphics[width=0.3\linewidth]{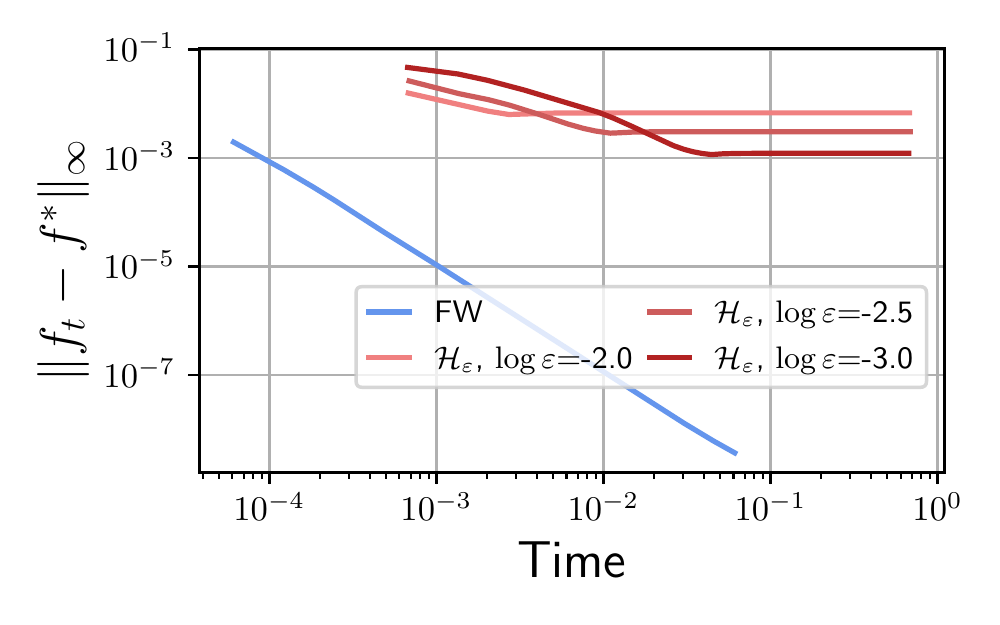}} &
		{\includegraphics[width=0.3\linewidth]{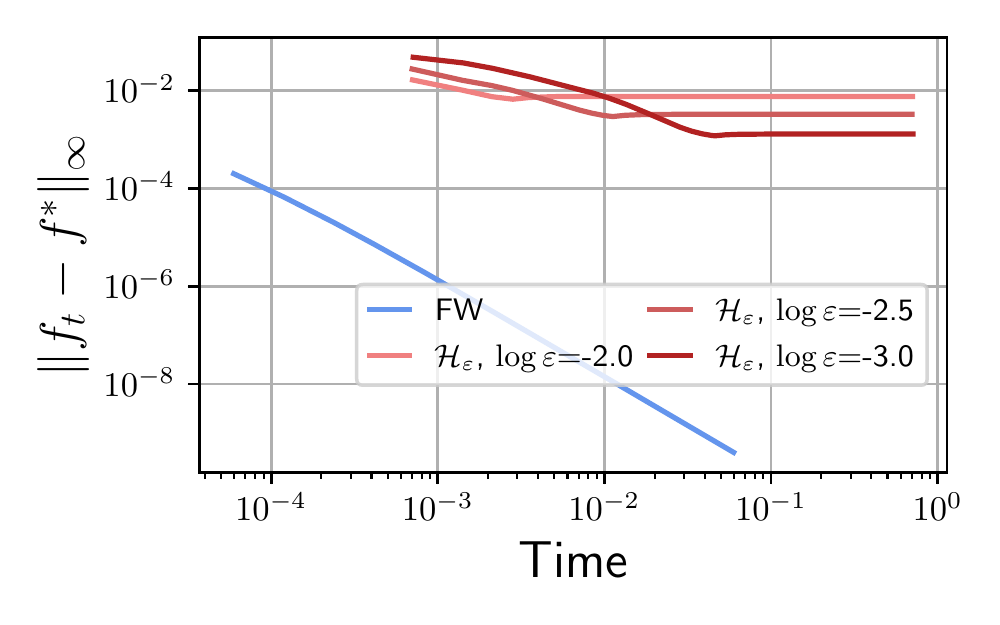}}
	\end{tabular}
	\caption{
		\textit{
			Same experiments as in Figure 4. The center plot is Figure 4 when $\rho=1$, while the left figure sets $\rho=10^{-1}$ and the right one sets $\rho=10$.
		} 
	}
\end{figure*}

\section{Appendix of Section 5 - Barycenters}
\label{seq-supp-barycenter}

\subsection{Proof of Proposition 6}
We provide in this section a different proof than that of the paper. It consists in completely rederiving the proof of~\cite{agueh-2011} for our functional $\UW$ which has an extra $\KL$ penalty term.

\begin{proof}
	First note that both problems have minima which are attained. Indeed, both problems admit finite values since respectively $\be=\al_1$ and $\ga=\al_1\otimes\ldots\otimes\al_K$ are feasible. 
	We can assume that the optimal plans have bounded mass.
	Assume for instance that it is not the case for the barycenter problem
	 Then there exists a sequence $\be_t$ approaching the infimum and such that $m(\be_t)\rightarrow\infty$, which would contradict the finiteness of the functional value. 
	Thus, we consider without loss of generality that $m(\be)<M_1$ and $m(\ga)< M_2$.
	By Banach-Alaoglu theorem, $\be$ and $\ga$ are in compact sets. Taking a sequence approaching the infimum, one can extract a converging subsequence which attains the minimum, hence the existence of minimizers.

	We prove now that the multimarginal problem upper-bounds the barycenter problem. 
	Take $\ga$ optimal for the multimarginal problem.
	Define the canonical projection $p_k$ such that $p_k(x_1,\ldots,x_K) \triangleq x_k$, and $\ga^{(k)} \triangleq (p_k, B_\la)_\sharp\ga$.
	Note that all $\ga^{(k)}$ have the same second marginal $\tbe\triangleq B_{\la\sharp}\ga$, thus they are feasible for $\UW(\al_k,\tbe)$, and we get
	\begin{align*}
	\UW(\al_k,\tbe)&\leq \dotp{\ga^{(k)}}{\C} + \rho\KL(\ga^{(k)}_1 | \al_k)\\
	&= \dotp{\ga}{\C[x_k, B_\la(x_1,\ldots,x_K)]} + \rho\KL(\ga_k | \al_k)
	\end{align*}
	Thus by summing over $k$, we get
	\begin{align*}
	\sum_k \la_k \UW(\al_k,\tbe) &\leq \sum_k \la_k\Bigg[  \dotp{\ga}{\C[x_k, B_\la(x_1,\ldots,x_K)]} + \rho\KL(\ga_k | \al_k)  \Bigg]\\
	&= \dotp{\ga}{\sum_k \la_k\C[x_k,  B_\la(x_1,\ldots,x_K)]} + \sum_k \la_k\rho\KL(\ga_k | \al_k)\\
	&= \dotp{\ga}{\Cc} + \sum_k \la_k\rho\KL(\ga_k | \al_k).
	\end{align*}
	Hence the upper-bound on the barycenter problem. 
	
	Now prove the converse inequality. Consider the optimal barycenter $\be^*$, and write $\pi^{(k)}(x_k, z)$ the optimal plan for $\UW(\al_k,\be^*)$.
	Note that all $(\pi^{(k)})$ have the same second marginal $\be^*$. 
	It allows to define the gluing of all plans along $\be^*$ which is a $(K+1)$-dimensional tensor, noted $\eta(x_1,\ldots,x_K,z)$. 
	Write $\teta$ its marginal/summation over the variable $z$. the plan $\teta(x_1,\ldots,x_K)$ is feasible for the multimarginal problem. 
	It yields
	\begin{align*}
	(2) &\leq \dotp{\teta}{\Cc} + \sum_k \la_k \rho\KL(\teta_k | \al_k)\\
	&= \dotp{\eta}{\Cc} + \sum_k \la_k \rho\KL(\eta_k | \al_k)\\
	&\leq \dotp{\eta(x_1,\ldots,x_K,z)}{\sum_k\la_k\C(x_k, z)} + \sum_k \la_k \rho\KL(\eta_k | \al_k)\\
	&= \sum_k \la_k \Bigg[  \dotp{\pi^{(k)}(x_k, z)}{\C(x_k, z)} +  \rho\KL(\pi^{(k)}_1 | \al_k) \Bigg].
	\end{align*}
	The last equality holds by construction of the gluing plan $\eta$, whose marginals with variables $(x_k,z)$ is $\pi^{(k)}$, which implies $\teta_k=\eta_k=\pi^{(k)}_1$. 
	The last line is exactly the value of the barycenter problem for the barycenter $\be^*$, which shows that it upper-bounds the multimarginal problem.
	
	Eventually, we have that both formulation yield the same value. 
	Reusing the first part of the proof, we have that the measure $\tbe = B_{\la\sharp}\ga$ yields the same value for both problems, thus it is an optimizer for the barycenter problem, which ends the proof.
\end{proof}

\subsection{Correctness of the OT multimarginal algorithm}
As explained, solving the 1D barycenter problem is equivalent to solving a balanced, 1D multimarginal transport problem with respect to the barycentric cost. 
For the Euclidean distance squared, it is well-known in the case of $K=2$ marginals that the barycenter is characterized by its generalized inverse cumulative distribution function (icdf) equal to the mean of the icdf of the corresponding marginals. This property can be extended in 1D to multimarginal costs satisfying 
a submodularity condition as defined in \cite{bach2019submodular, carlier2003class}.  Under this assumption, it is shown in these papers that the optimal plan $\gamma$, as defined in the corresponding multimarginal problem, is given by the distribution of $(F_{\alpha_1}^{-1}(U),\ldots,F_{\alpha_k}^{-1}(U))$ where $U$ is a uniform random variable on $[0,1]$ and $F^{-1}_{\mu}$ denotes the icdf of $\mu$ (see \cite{bach2019submodular} for the definition). 
Thus, the optimal primal variable $\gamma$ is explicitly parametrized by $t \in [0,1]$. 
It is direct to prove that the algorithm computes the optimal plan. The optimal dual variables are obtained by applying the primal-dual constraint which reads
\begin{equation}
\sum_{i = 1}^K f_i(x_{k_i}^i) = \Cc(x_{k_1}^1,\ldots,x_{k_K}^K)\, ,\label{EqPrimalDual}
\end{equation} 
for every $(x_{k_i}^i)_{i =1,\ldots,K}$ in the support of the optimal plan. To initialize the dual variables, one has to remark that the multimarginal OT problem is invariant by the following translations $f_i \to f_i + \lambda_i$ with $\sum_{i = 1}^K \lambda_i = 0$. This implies that one can set the value of the last $K-1$ potentials to $0$ and initialize $f_1(x_1^1) $  with the primal dual constraint~\eqref{EqPrimalDual} which gives $f_k(x_1^K) = \Cc(x_1^1,\ldots,x_1^K)$.
The standard iteration of the algorithm consists in updating the current point in the optimal plan (indexed by $t$) to the next point in the support of the plan $\ga$ and update the corresponding dual potential accordingly to the primal-dual constraint. 
Note that the primal-dual equality is satisfied on the support of the plan $\ga$ and on the support of $\al_1\otimes\ldots\otimes\al_K$ the inequality constraint
\begin{equation}
\sum_{i = 1}^K f_i(x_i) \leq \Cc(x_1,\ldots,x_K)\,
\end{equation} 
is satisfied, in other terms the potentials $(\f_1,\ldots,\f_K)$ is dual-feasible. This fact is guaranteed by a submodularity condition on the cost $\Cc$, and it is not satisfied for a general cost (see~\cite[Proposition 4]{bach2019submodular}).

\subsection{Proof of Proposition 8}
We provide here a more general proof where each marginal $\al_k$ is penalized by $\rho_k\KL$. We retrieve the result of the paper for the particular setting $\rho_k = \omega_k\rho$ and $\sum_k \omega_k = 1$.
We define $\rho_{tot}\triangleq \sum_k \rho_k$.

\begin{proof}
	First note that $\sum_k \la_k = 0$ implies that $(\f_1+\la_1)\oplus\ldots\oplus(\f_K + \la_K)\leq\Cc$. 
	In what follows we replace the parameterization $(\la_1,\ldots,\la_K)$ by $(\la_1 -\La,\ldots,\la_K - \La)$ where $\La(\la_1,\ldots,\la_K) \triangleq \tfrac{1}{K}\sum_k \la_k$, such that the constraint $\sum_k \la_k=0$ is always satisfied.
	The first order optimality condition on $\Ff(\f_1 + \la_1 - \La,\ldots,\f_K + \la_K -\La)$ w.r.t the coordinate $\la_i$ reads
	\begin{align*}
	&\dotp{\al_i}{e^{-(\f_i + \la_i  -\La) / \rho_i}} - \sum_{k=1}^K \tfrac{1}{K}\dotp{\al_k}{e^{-(\f_k +\la_k - \La)/ \rho_k}} = 0\\
	\Leftrightarrow & \dotp{\al_i}{e^{-(\f_i + \la_i  -\La) / \rho_i}} = \sum_{k=1}^K \tfrac{1}{K}\dotp{\al_k}{e^{-(\f_k +\la_k - \La)/ \rho_k}}\\
	\Leftrightarrow & \frac{-\la_i + \La}{\rho_i} + \log\dotp{\al_i}{e^{-\f_i  / \rho_i}} = \log\Bigg[\sum_{k=1}^K \tfrac{1}{K}\dotp{\al_k}{e^{-(\f_k +\la_k - \La)/ \rho_k}}\Bigg]\triangleq v\\
	\Leftrightarrow &-\la_i + \La + \rho_i \log\dotp{\al_i}{e^{-\f_i  / \rho_i}} = \rho_i\cdot v.
	\end{align*}
	Summing those optimiality equations for all $i$, one has $\sum_k (\la_k - \La) = 0$, thus yielding
	\eq{
		\rho_{tot}\cdot v = \sum_k \rho_k \log\dotp{\al_k}{e^{-\f_k / \rho_k}}.
	}
	Hence we get
	\eq{
		v = \frac{1}{\rho_{tot}} \sum_k \rho_k \log\dotp{\al_k}{e^{-\f_k / \rho_k}}.
	}
	Reusing the optimality condition, we set
	\eq{
		\la_i = \rho_i\log\dotp{\al_i}{e^{-\f_i / \rho_i}} - \frac{\rho_i}{\rho_{tot}} \sum_{k=1}^K \rho_k\log\dotp{\al_k}{e^{-\f_k / \rho_k}}.
	}
	Note that this formula verifies $K\La = \sum_k \la_k = 0$. Setting $\tal_k = e^{-(f_k + \la_k) / \rho_k}\al_k$, one has
	\eq{
		m(\tal_k) = \exp\Big(- \frac{1}{\rho_{tot}} \sum_k \rho_k \log\dotp{\al_k}{e^{-\f_k / \rho_k}}\Big).
	}
	Hence the equality of masses $m(\tal_i) = m(\tal_j)$ for any $(i,j)$.
\end{proof}

\end{document}